\documentclass[a4paper, 11pt]{amsart}

\usepackage{lmodern}
\usepackage[english]{babel}
\usepackage{amssymb, amsfonts, amsthm, amsmath, latexsym, esint, mathrsfs}
\usepackage{enumerate, caption, array, hyperref, enumitem}
\usepackage{tikz-qtree, subfloat}

\usepackage[margin=1.1in]{geometry}
\setlist[itemize]{leftmargin=*}

\usepackage{tikz}
\usetikzlibrary{calc,decorations.markings}
\usepackage[normalem]{ulem}

\theoremstyle{plain}
\newtheorem{theorem}{Theorem}[section]

\newtheorem{lemma}[theorem]{Lemma}
\newtheorem{proposition}[theorem]{Proposition}

\theoremstyle{definition}

\theoremstyle{remark}
\newtheorem*{remark}{Remark}

\numberwithin{equation}{section}

\newcommand\numberthis{\stepcounter{equation}\tag{\theequation}}

\newcommand{\ssum}[1]{\sum_{\substack{#1}}}

\newcommand{\e}{{\rm e}}
\newcommand{\dd}{\mathop{}\!\mathrm{d}}
\newcommand{\ee}{{\varepsilon}}

\newcommand{\C}{{\mathbb C}}
\newcommand{\E}{{\mathbb E}}
\newcommand{\D}{{\mathbb D}}
\newcommand{\N}{{\mathbb N}}

\renewcommand{\P}{{\mathbb P}}
\newcommand{\R}{{\mathbb R}}
\renewcommand{\S}{{\mathbf S}}
\newcommand{\Z}{{\mathbb Z}}
\newcommand{\1}{{\mathbf 1}}
\newcommand{\0}{{\mathbf 0}}

\newcommand{\cC}{{\mathcal C}}

\newcommand{\clF}{{\mathcal F}}
\newcommand{\cH}{{\mathcal H}}
\newcommand{\I}{{\mathcal I}}

\newcommand{\cU}{{\mathscr U}}
\newcommand{\cV}{{\mathcal V}}
\newcommand{\cW}{{\mathcal W}}

\newcommand{\F}{{\mathcal F}}
\newcommand{\FI}{{\mathcal F}_\triangleright}
\newcommand{\B}{{\mathcal B}}
\newcommand{\BI}{{\mathcal B}_\triangleright}

\let\H\relax
\DeclareMathOperator{\H}{{\mathbb H}}

\DeclareMathOperator{\U}{\mathbb{U}}
\DeclareMathOperator{\G}{{\mathcal G}}

\newcommand{\hf}{\tfrac12}
\newcommand{\hlf}{\frac12}

\DeclareMathOperator{\Id}{Id}
\DeclareMathOperator{\im}{Im}

\DeclareMathOperator{\QM}{?}
\DeclareMathOperator{\Res}{Res}
\DeclareMathOperator{\srd}{srd}
\DeclareMathOperator{\id}{id}

\DeclareMathOperator{\Prob}{{\mathbb P}}
\DeclareMathOperator{\Exp}{{\mathbb E}}
\DeclareMathOperator{\Var}{{\mathbb V}}

\renewcommand{\tilde}{\widetilde}

\renewcommand{\hat}{\widehat}

\newcommand{\psih}{{\hat\psi}}

\newcommand{\floor}[1]{{\left\lfloor {#1} \right\rfloor}}

\renewcommand\Re{\operatorname{Re}}

\newcommand{\ds}{\displaystyle}

\title{Statistical distribution of the Stern sequence}
\date{\today}


\author{S. Bettin}
\address{SB: Dipartimento di Matematica, Universit\`a di Genova, via Dodecaneso 35, 16146 Genova, Italy}
\email{bettin@dima.unige.it}

\author{S. Drappeau}
\address{SD: Aix Marseille Universit\'e, CNRS, Centrale Marseille, I2M UMR 7373, 13453 Marseille, France}
\email{sary-aurelien.drappeau@univ-amu.fr}

\author{L. Spiegelhofer}
\address{LS: Institut f\"ur Diskrete Mathematik und Geometrie, TU Wien, Wiedner Hauptstrasse 8--10, 1040 Wien, Austria}
\email{lukas.spiegelhofer@tuwien.ac.at}

\begin{document}

\begin{abstract}
We prove that the Stern diatomic sequence is asymptotically distributed according to a normal law, on a logarithmic scale. This is obtained by studying complex moments, and the analytic properties of a transfer operator. 
\end{abstract}

\maketitle

\section{Introduction}

The Stern diatomic sequence~\cite{Stern} is the sequence defined by the particularly simple recurrence relation
\begin{equation}
s(0)=0,\qquad s(1)=1, \qquad s(2n) = s(n), \qquad s(2n+1) = s(n)+s(n+1)\label{eq:recur-stern-intro}
\end{equation}
for all~$n\geq 1$. The first few terms are
$$ (s(n))_{n\geq 0} = (0, 1, 1, 2, 1, 3, 2, 3, 1, 4, 3, 5, 2, 5, 3, 4, 1, \dotsc). $$
It is an example of a~$2$-regular sequence~\cite[Chapter~16, Exercise~32]{AS-Automatic}, and enjoys various connections with mathematical objects. As important examples, the Stern sequence is related to the Stern-Brocot and the Calkwin-Wilf tree.

Starting from the set~$F_0=\{\tfrac01, \tfrac11\}$, for each~$N\in \N_{\geq 0}$, let~$F_{N+1}$ be built from~$F_N$ by inserting, between any two consecutive fractions~$\tfrac ab$ and~$\tfrac cd$, its median~$\tfrac{a+c}{b+d}$.

\begin{figure}[h]
\centering
\begin{tikzpicture} 
\tikzstyle{old}=[color=darkgray]
\tikzstyle{new}=[]

\node[new] (a1) at (-4,4.5) {$\tfrac01$};
\node[new] (a2) at ( 4,4.5) {$\tfrac11$};

\node[old] (b1) at (-4,3) {$\tfrac01$};
\node[new] (b2) at ( 0,3) {$\tfrac12$};
\node[old] (b3) at ( 4,3) {$\tfrac11$};

\node[old] (c1) at (-4,1.5) {$\tfrac01$};
\node[new] (c2) at (-2,1.5) {$\tfrac13$};
\node[old] (c3) at ( 0,1.5) {$\tfrac12$};
\node[new] (c4) at ( 2,1.5) {$\tfrac23$};
\node[old] (c5) at ( 4,1.5) {$\tfrac11$};

\node[old] (d1) at (-4,0) {$\tfrac01$};
\node[new] (d2) at (-3,0) {$\tfrac14$};
\node[old] (d3) at (-2,0) {$\tfrac13$};
\node[new] (d4) at (-1,0) {$\tfrac25$};
\node[old] (d5) at ( 0,0) {$\tfrac12$};
\node[new] (d6) at ( 1,0) {$\tfrac35$};
\node[old] (d7) at ( 2,0) {$\tfrac23$};
\node[new] (d8) at ( 3,0) {$\tfrac34$};
\node[old] (d9) at ( 4,0) {$\tfrac11$};

\draw (a1) -- (b2) -- (a2);

\draw (b1) -- (c2) -- (b2);
\draw (b2) -- (c4) -- (b3);

\draw (c1) -- (d2) -- (c2);
\draw (c2) -- (d4) -- (c3);
\draw (c3) -- (d6) -- (c4);
\draw (c4) -- (d8) -- (c5);

\node (f0) at (-6,4.5) {$F_0$:};
\node (f1) at (-6,  3) {$F_1$:};
\node (f2) at (-6,1.5) {$F_2$:};
\node (f3) at (-6,  0) {$F_3$:};

\end{tikzpicture} 
\caption{The first four rows of the Stern-Brocot tree}
%\label{fig:sbtree}
\end{figure}
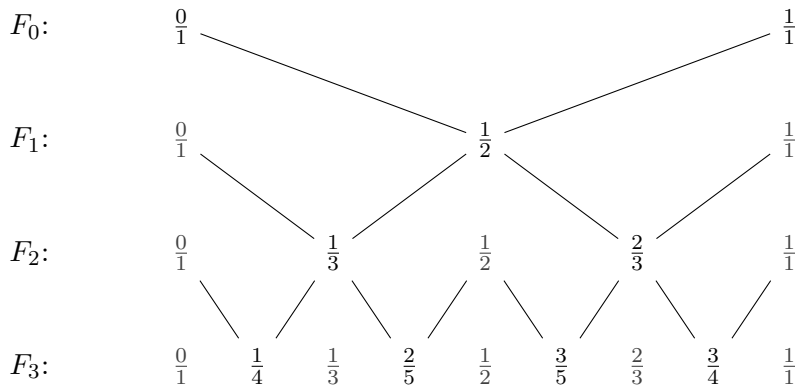

We may present the resulting construction as an infinite ``tree'', labelled by rationals in~$[0,1]$, which is known as the Stern-Brocot tree\footnote{Classical constructions start from~$\{\tfrac 01, \tfrac 10\}$, which makes little difference for our purposes.}: see for instance chapter 4.5 of~\cite{GKP} or section~1.5.1 of~\cite{KMS}.
The numerators and denominators appearing in~$F_N$, ordered by size of the fraction, are respectively the values of $s(m)$ and $s(m+2^N)$, for~$m\in[0,2^N]$.

Let~$\I_N = \Z \cap [2^N, 2^{N+1})$. In this paper, we will study properties of the values~$s(n)$ for~$n\in\I_N$: as we mentioned above, these are the denominators of the elements of the~$N$-th row of the Farey tree described above. Several properties of~$(s(n))_{n\in\I_N}$ were recorded by Stern~\cite{Stern} and Lehmer~\cite{Lehmer}.
There has been much interest in understanding the structure of the largest values of~$s(n)$~\cite{CT, Defant, Lansing, Paulin, CS}: as Lehmer showed, we have~$\max_{n\in\I_N} s(n) = F_{N+2}$, where~$(F_r)_r$ is the Fibonacci sequence. Recently, Paulin~\cite{Paulin} gave a complete description of the~$\floor{N/2}$ largest values taken by~$(s(n))_{n\in\I_N}$: they are given by various combinations of Fibonacci numbers.

\bigskip{}

Another interpretation of the Stern sequence can be obtained from the \emph{Calkin-Wilf} tree~\cite{CW}: it is the infinite binary tree, labelled by positive rationals in reduced form, starting from~$\frac11$, and where each node~$\frac ab$ has children~$\frac{a+b}b$ and~$\frac a{a+b}$. Each positive rational appears exactly once. In this case the denominators appearing at level $N$ are the values of $s(m+2^N)$ for~$m\in[0,2^N)$.

\begin{figure}[h]
\centering
\caption{First four rows of the Calkin-Wilf tree}
\label{fig:cwtree}
\tikzset{blank/.style={draw=none},
         edge from parent/.style=
         {draw,edge from parent path={(\tikzparentnode) -- (\tikzchildnode)}}}
\begin{tikzpicture}
\node {$\tfrac11$} [sibling distance=6cm]
child { node {$\tfrac21$} [sibling distance=3cm]
  child { node {$\tfrac31$} [sibling distance=1.5cm]
    child { node {$\tfrac41$} }
    child { node {$\tfrac34$} }
  }
  child { node {$\tfrac23$} [sibling distance=1.5cm]
    child { node {$\tfrac53$} }
    child { node {$\tfrac25$} }
  }
}
child { node {$\tfrac12$} [sibling distance=3cm]
  child { node {$\tfrac32$} [sibling distance=1.5cm]
    child { node {$\tfrac52$} }
    child { node {$\tfrac35$} }
  }
  child { node {$\tfrac13$} [sibling distance=1.5cm]
    child { node {$\tfrac43$} }
    child { node {$\tfrac14$} }
  }
};
\end{tikzpicture}
\end{figure}
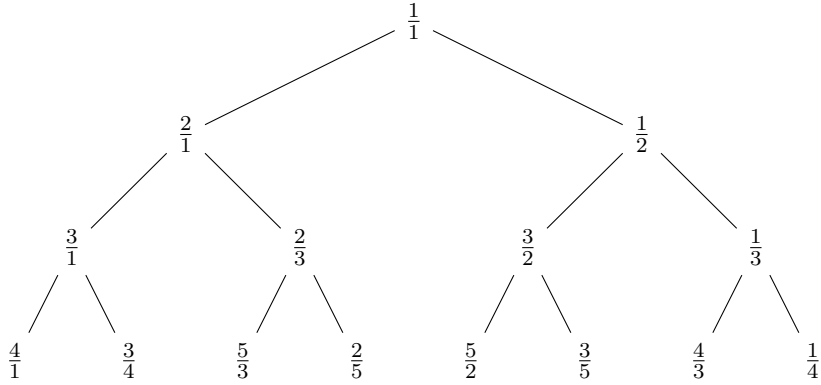

In the present paper, we are interested in the question of the \emph{statistical} distribution of values of~$s(n)$. A relevant setting consists in endowing, for each~$N\geq 0$, the finite set~$\I_N=\Z\cap[2^N, 2^{N+1})$ with the uniform probability measure; let~$\S_N$ be the random variable
$$ \S_N = s(n) $$
where~$n$ is taken uniformly randomly in~$\I_N$. Another way to look at this is the following: start from the root of the Calkin-Wilf, and follow a walk of~$N$ steps down the tree, choosing the left or right child with equal probability. Then the random variable~$\S_N$ is the denominator of the fraction eventually encountered.

Our main result is the following effective central limit theorem for~$\log\S_N$.

\begin{theorem}\label{th:distrib}
For some constants~$\alpha, \sigma>0$, as~$N$ tends to infinity, the values~$(\log s(n))$ are asymptotically distributed according to a Gaussian law, with mean~$\alpha N$ and variance~$\sigma^2 N$: for~$t\in\R$ satisfying~$t=O(N^{1/6})$, we have
\begin{equation}
\Prob_N\Biggl[\frac{\log\S_N - \alpha N}{\sigma \sqrt{N}} \leq t \Biggr]
= \int_{-\infty}^t \frac{\e^{-v^2/2}\dd v}{\sqrt{2\pi}} + O\bigg(\frac{(1+t^2)\e^{-t^2/2}}{\sqrt{N}}\bigg).\label{eq:estim-clt}
\end{equation}
Moreover, for some~$(\nu_1, \nu_2) \in\R^2$ and~$\theta\in[0,1)$, we have
\begin{align*}
\numberthis\label{eq:expect}   \Exp_N\bigl[ \log \S_N \bigr] {}& = \alpha N + \nu_1 + O(\theta^N), && (\text{\cite{Bacher}}) \hspace{-5em}\\
%\numberthis\label{eq:variance}
\Var_N\bigl[ \log \S_N \bigr] {}& = \sigma^2 N + \nu_2 + O(N\theta^N).
\end{align*}
\end{theorem}
\begin{remark}
\begin{itemize}
\item This answers a question posed in~\cite{Lansing}; the error term in~\eqref{eq:estim-clt} is relevant and optimal in the whole central limit range~$t=O(N^{1/6})$.
\item The constants~$\alpha$ and~$\sigma$ are numerically close to
$$ \alpha \approx 0.396212\dotsc, \qquad \sigma \approx 0.148905\dotsc $$
Both constants admit expressions as integrals of elementary functions with respect to a certain singular measure, the Minkowski measure (see formulas~\eqref{eq:alpha-log} and~\eqref{eq:sigma2-explicit}). They are studied in more detail in Section~\ref{sec:mean-value-alpha} below.
\item Formula~\eqref{eq:expect} was proved very recently by Bacher~\cite[Theorem~12.1]{Bacher} with a strong quantitative error term~$O(2^{-N})$. The approach used there does not require complex analysis, but is ineffective for other moments than the first.
\end{itemize}
\end{remark}

We will in fact prove that small complex moments of~$\S_N$ admit a quasi-powers expansion in the sense of Hwang~\cite{Hwang1,Hwang2}, meaning that $\Exp_N\bigl[(\S_N)^\tau\bigr]$ behaves asymptotically as~$A(\tau)^N B(\tau)$, for some holomorphic functions~$A, B$ in the neighborhood of the origin.

\begin{theorem}\label{th:moments}
For some~$\eta>0$ and~$\theta\in[0,1)$, there exist holomorphic functions~$U$, $V$ on the disc~$\{\tau\in\C:\ |\tau|\leq \eta\}$ such that
\begin{equation}
\Exp_N\bigl[(\S_N)^\tau\bigr] = \exp\big\{ N U(\tau) + V(\tau)\big\}\big(1 + O(\theta^N)\big),\label{eq:estim-moments}
\end{equation}
uniformly in $N\geq 1$ and $|\tau|\leq \eta$. Moreover,~$U''(0)\neq 0$.
\end{theorem}

This kind of estimates gives rise to the asymptotic normal law described in Theorem~\ref{th:distrib}, with speed of convergence, but also to bounds on probabilities of large deviations, and asymptotic formulae for~$\Exp[(\log \S_N)^k]$, with error term~$O_k(N^{k-1}\theta^N)$, for any fixed~$k\geq 1$. We refer to the above-quoted papers of Hwang~\cite{Hwang1,Hwang2}, and to Chapter IX.5 of~\cite{FS}, for more explanations.

In principle, explicit estimates for~$\theta$ could be obtained numerically. Experiments seem to suggest that any fixed~$\theta>1/2$ is admissible, at the cost of reducing the value of~$\eta$ accordingly.

\section{Overview}

To place Theorem~\ref{th:distrib} into context, we quote from~\cite{CW} the fact that~$(s(n+1)/s(n))_{n\in \I_N}$ is the~$N$-th row of the Calkin-Wilf tree (Figure~\ref{fig:cwtree}). It is easy to deduce from this the expression
\begin{equation}
\begin{pmatrix} s(n+1) \\ s(n) \end{pmatrix} = A_{\ee_0} \dotsb A_{\ee_{N-1}} \begin{pmatrix} 1 \\ 1 \end{pmatrix}\label{eq:stern-product-matrices}
\end{equation}
where~$n = 2^N + \sum_{j=0}^{N-1} \ee_j 2^j$, $\ee_j\in\{0, 1\}$, $A_0 = \left(\begin{smallmatrix} 1 & 1 \\ 0 & 1 \end{smallmatrix}\right)$ and~$A_1 = \left(\begin{smallmatrix} 1 & 0 \\ 1 & 1 \end{smallmatrix}\right)$. Picking~$n\in\I_N$ at random means choosing~$\ee_j\in\{0, 1\}$ independently with equal probability. This is an instance of a problem about random products of matrices, which is a vast and active area of research; we quote the seminal papers~\cite{Bellman, Furstenberg}, and we direct to the recent monograph~\cite{BQ-book} for more references. A general result such as Theorem~1.1 of~\cite{BQ} would indeed yield a slightly weaker version of our estimate~\eqref{eq:estim-clt}. We also refer to~\cite{EH} for a recent work concerned with a related situation (a ``law of large numbers'' for measures satisfying a recurrence relation similar to~\eqref{eq:recur-stern-intro}).

In order to obtain our precise statements, in particular Theorem~\ref{th:moments}, we rely on similar pools of ideas; however we will take a different approach, and cast the arguments in a more direct form, by exploiting a connection between the Stern sequence and the Minkowski question-mark function~\cite[Fig.7]{Minkowski}. As we will see shortly, this function arises naturally in our problem as a conjugacy between the two dynamical systems underlying the recursion formulas~\eqref{eq:recur-stern-intro}: the binary map and the Farey map, describing respectively the transformation rules of~$n$ and~$s(n)$. Let us first describe it following~\cite{Denjoy}, as the function mapping an irrational number~$x\in(0, 1]$ to
\begin{equation}\label{eq:def-minkowski}
\QM(x) := \sum_{n=1}^\infty \frac{(-1)^{n+1}}{2^{a_1+\dotsb+a_n}}, \qquad \text{if } x=[0;a_1,a_2,\dotsc] = \frac1{a_1+\frac1{a_2+\frac1{\dotsb}}}.
\end{equation}
This function is extended as a strictly increasing, continuous bijection from~$[0, 1]$ to itself. Let us define for notational convenience
$$ \Psi(x) := \QM(x), \qquad \Phi(x) := \QM^{-1}(x). $$
These functions are drawn in Figures~\ref{fig:graph-minkowski} and~\ref{fig:graph-conway}.
\begin{subfigures}
\begin{figure}[h]
  \centering
  \begin{minipage}{0.45\textwidth}
    \centering
    \includegraphics[width=0.8\textwidth]{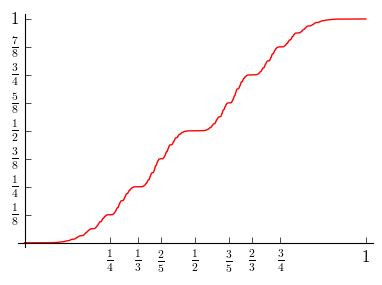}
    \caption{Graph of the Minkowski function~$\Psi(x) = \QM(x)$.}
    \label{fig:graph-minkowski}
  \end{minipage}\hfill
  \begin{minipage}{0.45\textwidth}
    \centering
    \includegraphics[width=0.8\textwidth]{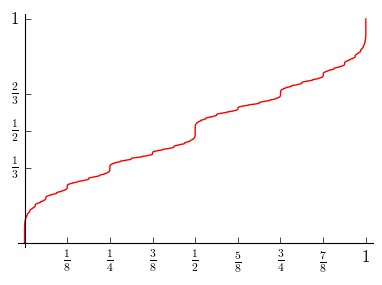}
    \caption{Graph of Conway's box function~$\Phi(x) = \QM^{-1}(x)$.}
    \label{fig:graph-conway}
  \end{minipage}
\end{figure}
\end{subfigures}

The question-mark function was introduced by Minkowski for its properties of mapping rational and quadratic irrational numbers into, respectively, dyadic and non-dyadic rational numbers. By construction it is increasing and continuous, but it is not absolutely continuous with respect to the Lebesgue measure, and in fact it is singular~\cite{Salem}: it has zero derivative for Lebesgue-almost every~$x\in[0, 1]$. It is however H\"older continuous with Lipschitz exponent $\frac{\log 2}{\sqrt 5+1}$.

An important property for us is the fact that the Minkowski function is a topological conjugation between two maps from~$[0, 1]$ to itself, the Farey map and the binary map, which are implicitly at play in the recurrence relation~\eqref{eq:recur-stern-intro}. We will detail this further below in Section~\ref{sec:heuristic-theorem}; a very practical form of this fact is given by the identity~\cite[Proposition~2.1]{Bacher}
$$ \frac{s(m)}{s(2^N+m)} = \Phi\Big(\frac m{2^N}\Big) $$
valid for all~$N\geq 0$ and~$m\in\{0, 1, \dotsc, 2^N\}$.
This formula will be our starting point: it provides a fast algorithm to express~$s(n)$ for~$n\in\I_N$ in terms of product of values of~$\Phi$, and brings the problem into the framework of dynamical analysis of Euclidean algorithms~\cite{Vallee, BV}.

In the rest of this section, we will first gather some facts about the mean-value~$\alpha$ (Section~\ref{sec:mean-value-alpha}), then we provide a naive heuristic towards Theorem~\ref{th:distrib} (Section~\ref{sec:heuristic-theorem}), and finally we outline our proof (Section~\ref{sec:outline-proof}).

\subsection{\texorpdfstring{The mean-value~$\alpha$ and the variance~$\sigma^2$}{The mean-value and the variance}}\label{sec:mean-value-alpha}

The constants~$\alpha$ and~$\sigma$ appearing in the mean-value estimate~\eqref{eq:expect} admit explicit expressions in terms of the Minkowski function. Let~$\mu$ denote the \emph{Minkowski measure} on the interval~$[0, 1]$,
\begin{equation}
\dd \mu(x) = \dd\Psi(x) = \dd(\Phi^{-1}(x)).\label{eq:def-mu}
\end{equation}
Then the constant $\alpha$ is given by
\begin{equation}
\alpha = -\frac12 \int_0^1 (\log x) \dd\mu(x).\label{eq:alpha-log}
\end{equation}
The next theorem records and links several different expressions of the constant~$\alpha$ and~$\sigma$.

\begin{theorem}\label{th:alpha}
\begin{enumerate}[label=(\roman*)]
\item Let~$m_k := \int_0^1 x^k \dd\mu(x)$. Then
\begin{equation}\label{bacher_alpha}
\alpha=\log 2-\sum_{k=1}^\infty\frac{m_k}{k2^k};
\end{equation}
this agrees with the definition given in~\cite[Th\'eor\`eme~12.1]{Bacher}.

\item We have
\begin{equation*}\label{eq:kinney_alpha}
\alpha = \int_0^1 \log(1+x)\dd\mu(x);
\end{equation*}
therefore the~$\mu$-almost-sure Lipschitz exponent~$\beta$ of~$x\mapsto \Psi(x)$, given in~\cite{Kinney}, is related to~$\alpha$ by~$\beta = \frac{\log 2}{2\alpha}$.

\item Consider the system~$([0,1], \clF, \mu)$ given by the Farey map~$\clF(x) = \min\{\tfrac x{1-x}, \tfrac{1-x}x\}$. Then we have
\begin{equation}
2\alpha = \int_0^1 \log\left|\clF'(x)\right|\dd\mu(x); \label{eq:alpha-entropy}
\end{equation}
in this form the constant~$2\alpha$ has an interpretation as a relative metric entropy.

\item\label{list:alpha-lambda1} The constant~$\alpha$ is the maximal Lyapunov exponent~$\lambda_1$ of the measure~$\eta$ on~$GL_2(\R)$ defined by~$\eta = \frac12\delta_{A_0}  + \frac12\delta_{A_1}$, where $A_0 = \left(\begin{smallmatrix} 1 & 1 \\ 0 & 1 \end{smallmatrix}\right)$ and~$A_1 = \left(\begin{smallmatrix} 1 & 0 \\ 1 & 1 \end{smallmatrix}\right)$. An~$\eta$-stationary measure~$\xi$ on~$\P(\R^2) = \R\cup\{\infty\}$ is given by
$$ \dd\xi(x) =  \tfrac12\mu + \tfrac12 T^*\mu\qquad (T(x) = 1/x), $$
where $\mu$ is extended trivially to a measure on $\R\cup\{\infty\}$ with support in $[0,1]$.

\item\label{list:logrho-explicit} The variance~$\sigma^2$ can be expressed as
\begin{equation}
\sigma^2 = \frac12\int_0^1 \bigg(\log x + \alpha\lfloor\tfrac1x\rfloor + \int_0^1 \log\bigg(\frac{1+y\{\tfrac1x\}}{1+yx}\bigg)\dd\mu(y)\bigg)^2 \dd\mu(x).\label{eq:sigma2-explicit}
\end{equation}

\end{enumerate}
\end{theorem}

\begin{remark}
The above facts and comments are detailed and proven in Section~\ref{sec:alpha}, where we will also show that the results of~\cite{Kinney} can be easily deduced from Theorem~\ref{th:distrib}. The main point in~\ref{list:alpha-lambda1} is the explicit expression for the stationary measure, which allows one to compute~$\alpha$ using a formula of Furstenberg.

Bacher~\cite{Bacher} showed that the moments~$m_k$ can be computed very accurately, allowing him to obtain a 50-digits approximation for~$\alpha$.
\end{remark}

\subsection{Iterates of the binary map, and a heuristic toward Theorem~\ref{th:distrib}}\label{sec:heuristic-theorem}

In this section we derive a formula for~$s(n)$ in terms of the function~$\Phi$, and then state a naive heuristic towards the fact that~$N^{-1} \log \S_N \to \alpha$ in law as~$N\to\infty$.

Consider the function~$\Phi$ defined on the dyadic rationals of~$[0, 1]$ by the formula
\begin{equation}
\Phi\Big(\frac m{2^N}\Big) = \frac{s(m)}{s(m+2^N)}, \qquad \Big(\frac{m}{2^N} \in [0, 1]\Big).\label{eq:link-phi-stern}
\end{equation}
As we have mentioned earlier, the function~$\Phi$ is given by the inversion bijection of Minkowski's function; but let us ignore this for a moment. The recurrence relations~\eqref{eq:recur-stern-intro} are equivalent to the facts that:
\begin{enumerate}
\item The definition~\eqref{eq:link-phi-stern} is well-posed, \emph{i.e.}, it genuinely only depends on the ratio~$m/2^N$,
\item Whenever~$N\geq 0$ and~$0\leq m < 2^N$, we have
\begin{equation}
\Phi\Big(\frac12\Big(\frac m{2^N} + \frac{m+1}{2^N}\Big)\Big) = \frac{s(m)+s(m+1)}{s(m+2^N)+s(m+1+2^N)}.\label{eq:Phi-conj-dyadic}
\end{equation}
\end{enumerate}
Additionally, we recall the fact that the Farey tree described in Figure~\ref{fig:cwtree} enumerates all rational numbers on~$[0, 1]$. Since, by~\eqref{eq:Phi-conj-dyadic}, the function~$\Phi$ is increasing, and~$(\Phi(0), \Phi(1))=(0,1)$, we obtain that~$\Phi$ may be extended to a homeomorphism from~$[0, 1]$ to itself. By these facts, a result of Panti~\cite[Proposition~1.1]{Panti} guarantees that~$\Phi$ must be the inverse of the Minkowski function.

We now iterate the relation~\eqref{eq:link-phi-stern}, in order to express~$s(2^N+m)$ in terms of values of~$\Phi$. Given~$1\leq m < 2^N$, there is a unique~$N'\in\N_{\geq 0}$, with~$N' < N$, such that~$2^{N'} \leq m < 2^{N'+1}$. Write~$m = 2^{N'+1} - m'$, so that~$1\leq m' \leq 2^{N'}$. It was noted by Stern~\cite{Bacher} (see~\cite{Lehmer} for an account) that~$s(m) = s(2^{N'+1}-m') = s(2^{N'}+m')$. Therefore, by~\eqref{eq:link-phi-stern}, we deduce
\begin{equation}
s(2^N+m) = \Phi\Big(\frac m{2^N}\Big)^{-1} s(2^{N'} + m').\label{eq:link-phi-stern-m'}
\end{equation}
It is easily checked that the map~$\BI : m/2^N \mapsto m'/2^{N'}$ is given by
\begin{equation}
\BI(x) = \begin{cases} 2-2^k x & \textnormal {for }x\in [2^{-k}, 2^{-k+1}),\ k\geq1 , \\ 0 &  \textnormal {for }x\in\{0, 1\}. \end{cases}\label{eq:def-map-jump-bin}
\end{equation}
This map is shown on the right, in Figure~\ref{fig:graph-bin}. It is the jump transformation of the binary (or ``tent'') map~$\B(x) = 2\min(x, 1-x)$ on the interval $[\frac12,1]$ (see~\cite{KMS}, fig. 1.8, page 46). The map~$\B$ is shown in Figure~\ref{fig:graph-bin-0}.

\begin{subfigures}
\begin{figure}[h]
  \centering
  \begin{minipage}{0.45\textwidth}
    \centering
    \includegraphics[width=0.8\textwidth]{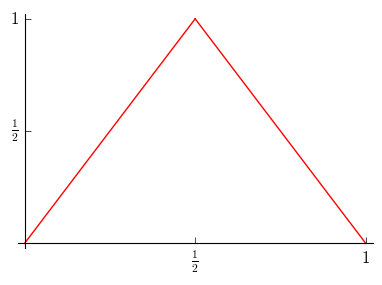}
    \caption{Graph of the binary map~$x\mapsto \B(x)$.}
    \label{fig:graph-bin-0}
  \end{minipage}\hfill
  \begin{minipage}{0.45\textwidth}
    \centering
    \includegraphics[width=0.8\textwidth]{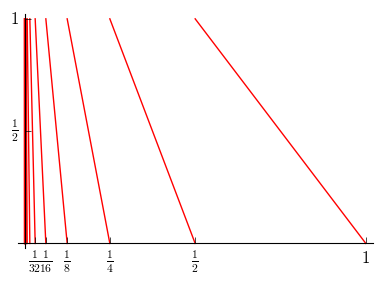}
    \caption{Graph of the jump transformation~$x\mapsto \BI(x)$.}
    \label{fig:graph-bin}
  \end{minipage}
\end{figure}
\end{subfigures}

Therefore, for any dyadic number~$m/2^N\in(0, 1]$, if we denote by~$K = K(m/2^N)\geq 0$ the least integer such that~$\BI^K(m/2^N) = 1$, then by iterating~\eqref{eq:link-phi-stern-m'}, we obtain
\begin{equation}
s(2^N+m) = \Bigg(\Phi\Big(\frac m{2^N}\Big) \Phi\Big(\BI\Big(\frac m{2^N}\Big)\Big) \dotsb \Phi\Big(\BI^{K-1}\Big(\frac m{2^N}\Big)\Big) \Bigg)^{-1}.\label{eq:stern-binary}
\end{equation}

The formula~\eqref{eq:stern-binary} has the advantage that the points~$m/2^N$ are easily described, however, the right-hand side involves the function~$\Phi$. At this point we may use the property that the function~$\Phi$ conjugates the binary map with another simple map: by Proposition~1.1 of~\cite{Panti}, we have the conjugacy relation
\begin{equation*}
\B = \Phi^{-1} \circ \F \circ \Phi \label{eq:conj-0}
\end{equation*}
between the binary map~$\B$, which we have already mentioned, and the \emph{Farey map} on the interval~$[0, 1]$, defined by
\begin{equation}\label{eq:def-Farey}
\F(x) = \min\Bigl(\frac x{1-x}, \frac{1-x}x\Bigr).
\end{equation}
By induction, we obtain the relation
\begin{equation}
\BI = \Phi^{-1} \circ \FI \circ \Phi \label{eq:conj}
\end{equation}
between the map~$\BI$, defined at~\eqref{eq:def-map-jump-bin}, and the ``jump transformation'' of the Farey map on the interval~$[\hlf,1]$, which is known (see \emph{e.g.}~\cite{PS, Isola}) to be precisely the \emph{Gauss map}
$$ \FI(x) = \G(x) := \begin{cases} 1/x - n & \text{for }x\in[\frac1{n+1}, \frac1n),\ n\geq 1, \\ 0 & \text{for } x\in\{0, 1\}. \end{cases} $$
\begin{subfigures}
\begin{figure}[h]
  \centering
  \begin{minipage}{0.45\textwidth}
    \centering
    \includegraphics[width=0.8\textwidth]{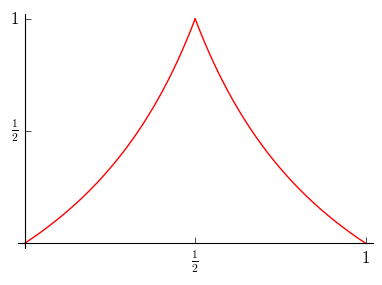}
    \caption{Graph of the Farey map~$x\mapsto \F(x)$.}
    \label{fig:graph-gauss-0}
  \end{minipage}\hfill
  \begin{minipage}{0.45\textwidth}
    \centering
    \includegraphics[width=0.8\textwidth]{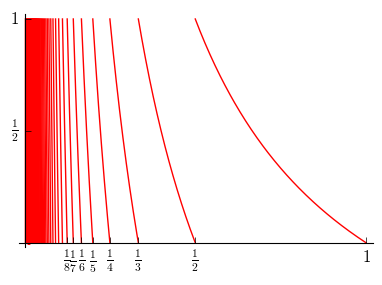}
    \caption{Graph of the Gauss map~$x\mapsto \G(x) = \FI(x)$.}
    \label{fig:graph-gauss}
  \end{minipage}
\end{figure}
\end{subfigures}

We mention at this point that the Lebesgue measure is invariant for the binary maps~$\B$ and~$\BI$. Upon conjugating, we deduce that the Minkowski measure~$\mu$ is invariant for the Farey and the Gauss maps; this fact will be used repeatedly in our arguments. Moreover, the measure~$\mu$ has maximal entropy~$\log 2$ for the Farey map; by contrast, the Farey map also admits~$\dd x/x$ as a unique invariant measure absolutely continuous with respect to the Lebesgue measure, but it is not finite, and has entropy zero. We refer to \cite{Lagarias} and to Chapter~1.2 of~\cite{KMS} for more details.

By~\eqref{eq:conj}, the relation~\eqref{eq:stern-binary} becomes the following.
\begin{proposition}\label{prop:stern-gauss}
Let $N\geq 0$, $1\leq m\leq 2^{N}$
and denote by~$K = K(m/2^N)\geq 0$ the least integer such that~$(\G^K \circ\ \Phi)(m/2^N) = 1$. Then
\begin{equation}
s(2^N+m) = \Bigg(\Phi\Big(\frac m{2^N}\Big) \G\Big(\Phi\Big(\frac m{2^N}\Big)\Big) \dotsb \G^{K-1}\Big(\Phi\Big(\frac m{2^N}\Big)\Big) \Bigg)^{-1}.\label{eq:stern-gauss}
\end{equation}
\end{proposition}

Equivalently, we have
\begin{equation*}
\log s(2^N+m) = -\sum_{i=0}^{K(m/2^N)-1}\log\Big((\G^i\circ\ \Phi)\Big(\frac m{2^N}\Big)\Big) \label{eq:stern-sumlogphi}
\end{equation*}
Our situation at this point is formally similar to the work of Baladi and Vall\'ee~\cite{BV}, concerned with statistical properties of orbits of rationals under the Gauss map. Our ``cost function'' here is~$x\mapsto -\log x$. This analogy provides a heuristic explanation for Theorem~\ref{th:distrib} as follows.

It is easily seen that for~$r\geq 0$ and~$N\geq 0$, $\bigl|\{m\in\{1, \dotsc, 2^N\}:\ K(\frac m{2^N})=r\}\bigr| = \binom Nr$. It follows that, as~$N\to\infty$, we have
$$ K(m/2^N) \sim \frac N2 $$
for a proportion~$1+o(1)$ of integers~$m\in\{1, \dotsc, 2^N\}$ as~$N\to\infty$ (which we abbreviate by ``generic $m$'').

On the other hand, by Theorem 5.12 of~\cite{MR-preimages}, we have that for a generic~$m\in\{1, \dotsc, 2^N\}$, the pre-images~$(\B^i(\frac m{2^N}))_{i=0}^N$ by the binary map will equidistribute as~$N\to\infty$ according to the Lebesgue measure on~$[0,1]$.
We may then guess that the same is true for the pre-images~$(\BI^i(\frac m{2^N}))_{i=0}^{K(m/2^N)}$ by the \emph{jump transformation $\BI$}. After conjugation by~$\Phi$, the Lebesgue measure is sent to the Minkowski measure~$\mu = \dd(\Phi^{-1})$: in particular, we are led to expect that for generic~$m\in\{1, \dotsc, 2^N\}$, we have
$$ \frac1{K(m/2^N)} \sum_{i=0}^{K(m/2^N)-1}\log\Big((\G^i\circ\ \Phi)\Big(\frac m{2^N}\Big)\Big) \sim \int_0^1 \log x \dd\mu(x), $$
and therefore for generic~$m\in\{1, \dotsc, 2^N\}$,
$$ \log s(2^N + m) \sim -\frac N2 \int_0^1 \log x \dd\mu(x). $$
This is indeed a consequence of Theorem~\ref{th:distrib} by our definition~\eqref{eq:alpha-log}.

This guess, which might seem naive at this stage, echoes a similar phenomenon for real, \emph{resp.} rational trajectories under the Gauss map~\cite{Hensley, BV}: see in particular the parallel between Theorems 1 and 3 of~\cite{BV}, and the factorization ``$\mu(c) = \mu \times {\hat \mu}(c)$'' stated there on page 350. For us, ``$\mu$'' plays the role the typical ratio of the length of rational trajectories~$K(m/2^N)$ by~$N$; and ``${\hat \mu}(c)$'' is the generic average of the cost function over real trajectories.

\subsection{Outline of the proof}\label{sec:outline-proof}

Having at hand the expression~\eqref{prop:stern-gauss} for the Stern sequence in terms of iterates of the Gauss map, the first step towards the proof of Theorem~\ref{th:moments} is to construct a generating series for the moment on the left-hand side of~\eqref{eq:estim-moments}. The precise form has to be amenable to analytic tools. For~$(\tau, z)\in\C^2$ of small enough moduli, let
\begin{align*}
S_\tau(z) {}& = \sum_{N\geq 0} (2z)^N \Exp_N[(\S_N)^{-\tau}] \numberthis\label{eq:def-MN-F}\\
{}& = \sum_{N\geq 0} z^N \sum_{n\in\I_N} s(n)^{-\tau}.
\end{align*}
The framework of analytic combinatorics~\cite{FS} relates the properties of~$S_\tau(z)$ (analytic continuation, meromorphy, location of poles, spectral gap) with the asymptotic behavior of~$\Exp_N[(\S_N)^{-\tau}]$.

The next step is to obtain the analytic information required on~$S_\tau(z)$. To this end, we adapt methods of ``dynamical analysis'', introduced by Vall\'ee and described for instance in~\cite{Vallee, BV}. The main point, which is behind the choice of the generating function~$S_\tau(z)$, is the expression
\begin{equation}
S_\tau(z) = \frac1{1-z} \sum_{k\geq 0}\Big(\H_{\tau, z}^k[\1]\Big)(1).\label{eq:Stz-sumH}
\end{equation}
This involves the iterates of an operator~$\H_{\tau, z}$, which acts on bounded functions~$f:[0, 1]\to\C$ by
$$ \H_{\tau,z}[f] : t \mapsto \sum_{n\geq 0} \frac{z^n}{(n+t)^\tau} f\Big(\frac1{n+t}\Big). $$
The operator~$\H_{\tau,z}$ is a particular case of a weighted Ruelle-Perron-Frobenius transfer operator~\cite{Ruelle, Baladi}. Our situation is similar to the work of Baladi-Vall\'ee~\cite{BV} on the Gaussian behavior of Euclidean algorithms. The pole of smallest modulus of~$S_\tau(z)$ will occur at a point~$z = \rho(\tau)$, where~$\H_{\tau,z}$ has dominant eigenvalue~$1$, and the Cauchy formula transfers this information into the estimate~\eqref{eq:estim-moments} with~$U(\tau) = -\log(2\rho(-\tau))$.

The deduction of Theorem~\ref{th:distrib} from Theorem~\ref{th:moments}, which is standard in probability theory, will be made by appealing to Hwang's Quasi-Powers theorem~\cite{Hwang2}. We remark that the uniformity of Theorem~\ref{th:moments} is crucial in this deduction.

As in \cite{Morris}, we have chosen to include details of the arguments from spectral theory, rather than quote them as a black box, so that readers unacquainted with these topics, but who are still interested in the arithmetic application, may follow through.

\bigskip{}

The plan is the following: in Section~\ref{sec:moments-to-transferop}, we derive the expression~\eqref{eq:Stz-sumH}. In Section~\ref{sec:prop-transferop}, we study the operator~$\H_{\tau,z}$, with a particular emphasis on the reference pair~$(\tau,z)=(0,\frac12)$. In Section~\ref{sec:analysis-Stz}, we carry out the analysis of~$S_\tau(z)$. In Section~\ref{sec:proof-th}, we complete the proof of Theorems~\ref{th:distrib} and~\ref{th:moments}. Finally in Section~\ref{sec:alpha}, we return to the mean-value~$\alpha$ and the variance~$\sigma^2$, and prove Theorem~\ref{th:alpha}.

\section{Expressing the moment-generating function}\label{sec:moments-to-transferop}

In this section we express the moment-generating function on the left-hand side of~\eqref{eq:estim-moments} in terms of a weighted transfer operator for the Gauss map, continuing the arguments of Section~\ref{sec:heuristic-theorem}.

We recall Proposition~\ref{prop:stern-gauss}. Note that the Gauss map has the property that~$\G'(x) = -1/x^2$ for~$x>0$. It is therefore very convenient to use it in conjunction with~\eqref{eq:stern-gauss} to obtain the ``product of cocycles''
\begin{equation}\begin{aligned}
s(2^N+m)^2 {}& = \prod_{i=0}^{K-1}\left|\G'\right|\circ\G^i\Big(\Phi\Big(\frac m{2^N}\Big)\Big) \\
{}& = |(\G^K)'|\Big(\Phi\Big(\frac m{2^N}\Big)\Big).\label{eq:stern-gauss-derivative}
\end{aligned}\end{equation}

The inverse branches of the Gauss map~$\G$ form the set
$$ \cH = \{h_n,\ n\geq 1\}, \qquad h_n(x) = \frac1{n+x} \quad (x\in[0, 1]). $$
For any given~$K\in\N_{\geq 0}$, the set~$\cH^K$ of inverse branches of the function~$\G^K : [0, 1]\to[0, 1]$ are then given by
$$ \cH^K = \{ h_{n_1} \circ \dotsb \circ h_{n_K}:\ n_j\geq 1\}, $$
where it is understood that~$\cH^0 = \{\id\}$. The decomposition~$h=h_{n_1} \circ \dotsb \circ h_{n_K}$ is unique, therefore, it makes sense to define
$$ w(h) := n_1 + \dotsb + n_K. $$
An immediate verification shows that
$$ \{ \tfrac{m}{2^N} \ |\  0\leq m \leq 2^N, \ m\textnormal{ odd} \} = \B^{-N}(\{1\}) $$
and thus, letting~$\cH^* = \cup_{K\geq 0} \cH^K$,
\begin{equation}
\begin{aligned}
 \{ \Phi(\tfrac{m}{2^N}) \ |\  0\leq m \leq 2^N, \ m\textnormal{ odd} \} {}& = \F^{-N}(\{1\}) \\
{}& = \{ h(1) \ |\  h\in \cH^*, \ w(h) = N\}.  \label{eq:invF}
\end{aligned}
\end{equation}
Gathering the above, we obtain
\begin{align*}
S_\tau(z) {}& = \sum_{N\geq 0} z^N \sum_{m=1}^{2^N} s(m+2^N)^{-\tau} \\
{}& = \sum_{N\geq 0} z^N \sum_{0\leq r \leq N} \sum_{\substack{m=1 \\ m\text{ odd}}}^{2^r} s(m+2^r)^{-\tau} \\
{}& = \frac1{1-z}\sum_{r\geq 0} z^r \sum_{K\geq 0} \ssum{h\in\cH^K \\ w(h)=r} |(\G^K)' \circ h(1)|^{-\tau/2} \\
{}& = \frac1{1-z}\sum_{h\in\cH^*} z^{w(h)} |h'(1)|^{\tau/2}, \numberthis\label{eq:Stz-Hstar}
\end{align*}
where we have extracted the largest power of~$2$ dividing~$m$ in the first line; the third line used~\eqref{eq:stern-gauss-derivative} and~\eqref{eq:invF}, and the last line followed by the derivative formula for the inverse.

The sum over~$h\in\cH^*$ is now recognized as a sum of iterates of a ``density transformer'', or transfer operator, also called Ruelle-Perron-Frobenius operator. These objects, and their spectral properties, have a long history and have been extensively studied, notably in connection with continued fraction, and thermodynamic formalism. We refer to the surveys~\cite{Ruelle}, the lecture notes~\cite{Mayer} and the monographs~\cite{Baladi, Ruelle-thermo}. More explanations and references relevant to our case are found in Section 2.2 of~\cite{BV}. For all~$(\tau, z)\in\C^2$ with~$|z|<1$, let the operator~$\H_{\tau,z}$ act on continuous functions~$f\in\cC([0, 1])$ by
\begin{equation}\label{eq:def-Htz}
\begin{aligned}
\H_{\tau, z}[f](t) {}& = \sum_{h\in\cH} z^n |h'(t)|^{\tau/2} (f\circ h)(t) \\
{}& = \sum_{n\geq 1} \frac{z^n}{(n+t)^{\tau}} f\Big(\frac{1}{n+t}\Big).
\end{aligned}
\end{equation}
Then from formula~(2.7) of~\cite{BV}, we have that for all~$K\geq 0$
\begin{equation}
\H^K_{\tau,z}[f](t) = \sum_{h\in\cH^K} z^{w(h)} |h'(t)|^{\tau/2} (f\circ h)(t).\label{eq:iter-H}
\end{equation}
Let now
\begin{equation}
\cW = \{(\tau, z)\in\C^2:\ |z| < \min(1, 2^{-\Re \tau})\}.\label{eq:def-small-tau-z}
\end{equation}
By equations~\eqref{eq:Stz-Hstar}, \eqref{eq:def-Htz} and \eqref{eq:iter-H} above, we conclude the following.
\begin{proposition}\label{prop:Fsz-H-iter}
For~$(\tau, z)\in \cW$, we have
$$ S_\tau(z) = \frac1{1-z} \sum_{k \geq 0} \H_{\tau, z}^k[\1](1) = \frac 1{1-z}(\Id - \H_{\tau, z})^{-1}[\1](1). $$
\end{proposition}

\section{Properties of the transfer operator}\label{sec:prop-transferop}

As we have mentioned, spectral properties of transfer operators have been extensively studied; we refer again to the lecture notes~\cite{Mayer}, sections 7.1 and 7.4, and the references therein. The actual operator~$\H_{\tau,z}$ has been defined and studied at many occurrences in the literature: see~\cite[formula~(10)]{PS}, \cite{Dodds}, \cite[formula~(3.39)]{Isola}. Indeed, many of the forthcoming properties of~$\H_{\tau,z}$ for real~$z$ can be found in~\cite{PS}. Nonetheless, to make the arguments as clear as possible for readers unacquainted with these topics, we will provide full proofs, following the presentation of~\cite{Morris}, apart from perturbation theory of operators, which we will quote from~\cite{Kato}.

\subsection{Definitions}

A critical first step is to define an appropriate functional space in which to study~$\H_{\tau,z}$. For our arithmetic application, two constraints must be satisfied: the interval~$[0, 1]$ should be contained in their domain, and it must include the constant function~$\1$ and all its iterates under~$\H_{\tau,z}$ for all~$(\tau, z)$ in a neighborhood of the origin.

Consider the domain~$\D = \{t\in\C:\ |t-\frac23|<1\}$, and the set of functions
$$ H^\infty(\D) = \big\{f : \D \to \C:\ f\text{ is holomorphic and bounded}\}. $$
For~$(\tau, z) \in \cW$ and~$f\in H^\infty(\D)$, we define~$\H_{\tau,z}$ as in~\eqref{eq:def-Htz}, taking the principal determination of the logarithm.

For~${\mathbb T}$ an operator, we denote by~$\srd({\mathbb T})$ the spectral radius of~$\mathbb{T}$. We further recall the definition of the Gauss map and its inverse branches,
$$ \G(x) := \begin{cases} 1/x - n, & x\in[\frac1{n+1}, \frac1n), \\ 0, & x\in\{0, 1\}, \end{cases} \qquad\qquad h_n(x) := \frac1{n+x} . $$

\subsection{Decomposition of the transfer operator}

In this section, we obtain the basic properties of the operator $\H_{\tau,z}$. The arguments involved have a long history~\cite{Mayer}; we will mostly follow the presentation found in~\cite{Morris}, which is well adapted to our setting. We are interested in the behavior of $\Exp_N\bigl[(\S_N)^\tau\bigr]$ when $\tau$ is in a neighborhood of $\tau=0$. At this value the power series $S_\tau$ has radius of convergence $\frac12$, and thus of particular importance for us is the value~$(\tau,z)=(0,\frac12)$, around which~$\H_{\tau, z}$ will have~$1$ as an eigenvalue. We denote
$$ \H := \H_{0,\frac12}. $$
We will first focus on~$\H$, then on the case~$\tau=0$, $|z|=\frac12$, and finally we study arbitrary~$(\tau, z)$ using the theory of analytic perturbation (as in~\cite[Chapter IV]{Kato}). 

\begin{proposition}\label{prop:decomp-H}
\begin{enumerate}[label=(\roman*)]
\item For~$|z|<1$ and~$\tau\in\C$, the operator~$\H_{\tau,z}$ is compact and depends holomorphically on~$(\tau,z)$ in the sense of~\cite[p.366]{Kato}.
\item The eigenmeasure of the adjoint operator $\H^*$ is $\mu$, that is, for all~$f\in L^1(\mu)$ we have~$\H[f]\in L^1(\mu)$ and
\begin{equation}
\int \H[f] \dd\mu = \int f\dd\mu.\label{eq:eigenmeasure}
\end{equation} 
Moreover, the operator $\H$ has a simple isolated eigenvalue equal to $1$; the corresponding eigenspace is generated by the constant function $\1$. Finally, this eigenvalue is the only element  of modulus $1$ in the spectrum of $\H$.
\end{enumerate}
Furthermore,  for all small enough~$\delta_1>0$, we may find~$\delta_2>0$ and~$\theta\in[0,1)$ such that, writing
\begin{align*}
\cV_1 {}& := \{(\tau, z)\in\C^2:\ |\tau|\leq \delta_1, \ |z-\hf|\leq \delta_1\}, \\
\cV_2 {}& := \{(\tau, z)\in\C^2:\ |\tau|\leq \delta_2, \ |z|\leq \hf+\delta_2,\ |z-\hf|\geq \delta_1\},
\end{align*}
the following holds.
\begin{enumerate}[label=(\roman*)]
  \setcounter{enumi}{2}
\item For~$(\tau,z)\in\cV_1$, $\H_{\tau,z} $ has a simple and isolated dominant eigenvalue at $\lambda(\tau,z)\in\C$ with eigenfunction  $f_{\tau,z} \in H^\infty(\D)\smallsetminus\{0\}$; also,~$\lambda(0,\hlf)=1$ and $f_{0,\hlf}=\1$. In particular, we have the decomposition
\begin{equation}
\H_{\tau,z} = \lambda(\tau,z)\P_{\tau,z} + \N_{\tau,z},\label{eq:decomp-Hsz}
\end{equation}
with~$\P_{\tau,z}$ and~$\N_{\tau,z}$ two compact operators on~$H^\infty(\D)$. Moreover, the image of~$\P_{\tau,z}$ is one-dimensional: $\im\P_{\tau,z} = f_{\tau,z}\C$ and~${f_{\tau,z} = \P_{\tau,z}[\1]} $. Finally, we have~${\srd}(\N_{\tau,z})\leq \theta$ and~$\P_{\tau,z}\N_{\tau,z}=\N_{\tau,z}\P_{\tau,z}=\0$.

\item For~$(\tau,z)\in\cV_1$,~$f_{\tau,z}$ and~$\lambda(\tau,z)$ depend holomorphically on both variables.
\item For~$(\tau,z)\in\cV_2$, we have~$\srd({\H_{\tau,z}})\leq\theta$.
\end{enumerate}
\end{proposition}

\begin{proof}
{(i) Let $ \tilde\D:= \{t\in\C:\ |t-\frac23|<\frac{17}{16}\}$. We have that $h_n$ maps $\tilde\D$ into %a compact subset of 
$\D$ for all $n\geq1$. In particular, all the elements of the sets $S:=\{f\circ h_n\mid f\in H^\infty(\D),\, \|f\|_\infty\leq 1\}$ can be extended to functions in the unit disk of  $H^\infty(\tilde \D)$. By Montel's theorem the set obtained by these extended functions is pre-compact with respect to the compact-open topology on $H^\infty(\tilde\D)$ and thus $S$ is pre-compact with respect to the uniform topology of $H^\infty( \D)$. In particular, the operator $f\mapsto f\circ h_n$ on $ H^\infty(\D)$ is compact and thus $f\mapsto \frac{z^n}{(n+\cdot )^\tau}f\big(\frac1{n+\cdot }\big) $ is also compact (and holomorphic in $(\tau,z)$). The same then holds for $\H_{\tau,z}$ since its defining series converges locally uniformly in $(\tau,z)$.}

(ii) One has that $\H[f](x)=\U[f\circ \Phi](\Phi^{-1}(x))$, where $\U$ is the Ruelle operator relative to the binary map~$\BI$ defined at~\eqref{eq:def-map-jump-bin}; since the Lebesgue measure is an eigenmeasure for $\U^*$~\cite[Prop. 2.3.21]{KMS} with eigenvalue~$1$, it follows that $\mu$ is an eigenmeasure for $\H^*$ with the same eigenvalue. One could also verify this directly from the definition~\eqref{eq:def-minkowski}. Indeed, since~$\Psi(h_n(x)) = 2^{-n}(2-\Psi(x))$ for all~$n\in\N^*$ then one has~$\mu(A) = 2^{-n} \mu(h_n^{-1}(A))$ for any interval~$A = [\tfrac1{n+1}, \tfrac1{n+x}]$, with $x\in[0,1]$. By the monotone class lemma one also has~$\mu(A\cap[\tfrac1{n+1}, \tfrac1n])=2^{-n}\mu(h_n^{-1}(A))$ for all~$\mu$-measurable set~$A$, and finally~$2^{-n} \int_{[0,1]} (f\circ h_n)\dd\mu = \int_{[1/(n+1),1/n]} f\dd\mu$ for all~$f\in L^1(\mu)$. Summing over~$n\geq 1$ yields~\eqref{eq:eigenmeasure}.

Since $\H$ is a compact operator, the non-zero elements of its spectrum are isolated eigenvalues of finite multiplicity; also, notice that $\H[\1]=\1$. Let $f\in H^\infty(\D)$ be an eigenfunction with eigenvalue $\lambda$ normalized so that $\max_{t\in[0,1]}|f(t)| = |f(t_0)| = 1$ for some~$t_0\in[0,1]$. Then, by the definition of $\H[f]$ and the triangle inequality we have
$$ |\lambda|  = |\H[f](t_0)|= \Big|\sum_{n\geq 1} 2^{-n} f\Big(\frac1{n+t_0}\Big)\Big| \leq \sum_{n\geq 1} 2^{-n} = 1. $$
Thus, $|\lambda|\leq 1$. Also, if $|\lambda|=1$ then the equality holds everywhere and so there exists $c\in\C$ of modulus $1$ such that $f(\frac1{n+t_0})=c$ for all $n\geq1$. Since $f$ is holomorphic and $\frac1{n+t_0}$ has $0\in\D$ as an accumulation point, we must have $f=c\1$. Thus, $\lambda=1$ is the only eigenvalue of modulus $1$ and $\ker (\H-\Id)$ is $1$-dimensional. Moreover, if $(\H-\Id)^2[f]=\0$ then $(\H-\Id)[f]$ is an eigenfunction of $\H$ with eigenvalue $1$ and so $(\H-\Id)[f]=c\1$ for some $c\in\C$. Integrating this equation with respect to $\dd\mu$ by~\eqref{eq:eigenmeasure} we find $c=0$. Thus, $f$ is itself a multiple of $\1$ and so $1$ is a simple isolated eigenvalue  of $\H$. 

(iii)
Now, let $C$ be a small circle centered at $1$ which doesn't enclose any other eigenvalue of $\H$ and assume $(\tau,z)\in\cV_1$. If $\delta_1$ is small enough, then by~\cite[Thm~IV.3.16]{Kato} we have that $C$ doesn't intersect the spectrum of ${\srd}(\H_{\tau,z})$. It follows that we can write $\H_{\tau,z}$ as a sum of compact operators $\H_{\tau,z}=\P_{\tau,z}+\N_{\tau,z}$ with $\P_{\tau,z}\N_{\tau,z}=\N_{\tau,z}\P_{\tau,z}=0$ and $\P_{\tau,z}^2=\P_{\tau,z}$, where $\P_{\tau,z}$ is the Riesz projection associated to $C$~\cite[Thm III.6.17]{Kato}. Moreover, since $1$ is a simple eigenvalue, the image of $\P_{0,\hlf}$ is one dimensional~\cite[pp.180-181]{Kato} and thus the same holds for  $\P_{\tau,z}$ if $\delta_1$ is small enough~\cite[Thm~IV.3.16]{Kato}.
Also, the spectrum of {$\H_{\tau,z}$ restricted to the image of~$\P_{\tau,z}$} consists of a unique eigenvalue $\lambda(\tau,z)$,  with $\lambda(0,\hlf)=1$, corresponding to the eigenfunction $f_{\tau, z}:=\P_{\tau,z}\1$, whereas the spectrum of $\N_{\tau,z}$ consists of that of $\H_{\tau,z}$ with $\lambda(\tau,z)$ removed~\cite[Thm~IV.3.16]{Kato}. In particular since $\srd(\N_{0,\hlf})\leq \theta'<1$ for some $\theta'\in[0,1)$, then by the upper-semicontinuity of the spectral radius~\cite[Thm~IV.3.16]{Kato} there exists $\theta\in[0,1)$ such that ${\srd}(\N_{\tau,z})\leq \theta$ for all $(\tau,z)\in\cV_1$ with $\delta_1$ small enough. 

(iv) By~\cite[Thm VII.1.7]{Kato} $\P_{\tau,z}$ and $\N_{\tau,z}$ depend homomorphically on $(\tau,z)$ and thus so does $f_{\tau, z}=\P_{\tau,z}[\1]$. Moreover, since $\P_{0,\hlf}[\1](0)=1$ we have $f_{\tau, z}(0)\neq0$ for  $(\tau,z)\in\cV_1$
and $\delta_1$ small enough. Thus the holomorphicity of $\lambda(\tau,z)$ follows since $\P_{\tau,z}[f_{\tau, z}](0)=\lambda(\tau,z)f_{\tau, z}(0)$.

(v)  First, consider the case~$\tau=0$ and~$z=\frac12\e^{2\pi i\phi}$ with~$\phi\in[0,1)$.  By the triangle inequality we have~$\|\H_{0,z}[f]\|_\infty\leq \|f\|_\infty$ for all~$f\in H^\infty(\D)$; in particular~${\srd}(\H_{0,z})\leq1$ for~$|z|=\hf$. Suppose now~$\phi\neq 0$. Then since~$\H_{0,z}$ is compact, it has an eigenfunction $f\in H^\infty(\D)$ with eigenvalue $\lambda$ of maximum modulus~${\srd}(\H_{0,z})$. Up to re-scaling $f$, we can assume $\max_{t\in[0,1]}|f(t)| = |f(t_0)| = 1$ for some~$t_0\in[0,1]$, whence
$$| \lambda | = |\H_{0,z}[f](t_0)|= \Big|\sum_{n\geq 1} 2^{-n}e^{2\pi i n\phi} f\Big(\frac1{n+t_0}\Big)\Big| \leq \sum_{n\geq 1} 2^{-n} = 1. $$
If we had~$|\lambda|=1$, then this would mean equality holds everywhere. This implies that all the summands in the first series have the same argument and the $n$-summand has modulus $2^{-n}$, that is
$$ e^{2\pi i n\phi} f\Big(\frac1{n+t_0}\Big) = \e^{2\pi i \alpha}  $$
for some $\alpha\in[0,1)$. Letting~$n\to\infty$, we deduce~$\e^{2\pi i n \phi} \to f(0) $, which implies~$\phi=0$. We conclude that~${\srd}(\H_{0,z})<1$ whenever~$|z|=\hf$ and~$z\neq \hf$.
Then, by the upper semi-continuity of $\srd(\H_{\tau,z})$~\cite[Thm~IV.3.16]{Kato}, we have that there exists $0<\delta_2<\frac12$ such that for $\tau,z$ satisfying 
$$\hf-\delta_2\leq |z|\leq \hf+\delta_2,\quad |z-\hf|\geq \delta_1,\quad |\tau|\leq \delta_2$$
 we also have $\srd(\H_{\tau,z})\leq\theta<1$. Finally, assume $|z|\leq \hlf -\delta_2$ and $|\tau|\leq \delta_2<-\hlf\log_2(1-2\delta_2)$. Then, by the triangle inequality for any~$f\in H^\infty(\D)$, we have
$$ \| \H_{\tau,z}[f] \|_\infty \leq \|f\|_\infty \sum_{n\geq 1} |z|^n(n+1)^{\delta_2}. $$
Since~$n+1\leq 2^n$, computing the series proves that~${\srd}(\H_{\tau,z})<1$. In particular, $\srd(\H_{\tau,z})<1$ for all $(\tau,z)\in\cV_2$ and (iv) follows.
\end{proof}

\subsection{\texorpdfstring{Specific properties at~$(\tau,z)=(0,\hf)$}{Specific properties at tau=0, z=half}}

We have the following properties linking~$\H$ and~$\mu$. We recall that the value~$\lambda(0,\hf)=1$ was proved at Proposition~\ref{prop:decomp-H}.(ii).

\begin{proposition}\label{prop:measure}
\begin{enumerate}[label=(\roman*)]
\item We have~$\P_{0,\hlf}[f] = (\int f\dd\mu)\1$ for all~$f\in H^\infty(\D)$.
\item The derivatives of~$\lambda$ satisfy
$$ \tfrac{\partial}{\partial \tau}\lambda(0, \hf) = \int_{[0, 1]}\log\dd\mu, \qquad \tfrac{\partial}{\partial z}\lambda(0, \hf)=4. $$
\end{enumerate}
\end{proposition}

\begin{proof}
(i) First we recall that $f_{0,\hlf}=\1$ and thus $\im \P_{0,\hlf}=\C\,\1$.
Given~$f\in H^\infty(\D)$, we have~$f\in L^1(\mu)$ and so, by~\eqref{eq:eigenmeasure},~$\int f\dd\mu = \int \H^k[f]\dd\mu$ for all~$k\in\N$. Letting~$k\to\infty$, from the decomposition~\eqref{eq:decomp-Hsz} we have~$\H^k[f] \to \P_{0,\hlf}[f]$ uniformly on~$[0,1]$, and so~$\int \H^k[f]\dd\mu \to \int \P_{0,\hlf}[f]\dd\mu$. Since~$\P_{0,\hlf}[f]\in \C\,\1$ is constant, this proves our claim.

(ii) We have, following the notations of Proposition~\ref{prop:decomp-H}, $\lambda(\tau,z) f_{\tau,z} = \H_{\tau,z}[f_{\tau,z}]$ for~$(\tau,z)$ in a neighborhood of~$(0,\hf)$. By uniform convergence of the series defining $\H_{\tau,z}$, we may differentiate term-wise and obtain
$$ \def\arraystretch{1.5} \left\{\begin{array}{l} \ds \tfrac{\partial}{\partial \tau}\lambda(\tau,z) f_{\tau,z} + \lambda(\tau,z) \tfrac{\partial}{\partial \tau}f_{\tau,z} = \H_{\tau,z}[f_{\tau,z}\log + \tfrac{\partial}{\partial \tau}f_{\tau,z}], \\
\ds \tfrac{\partial}{\partial z}\lambda(\tau,z) f_{\tau,z} + \lambda(\tau,z) \tfrac{\partial}{\partial z}f_{\tau,z} = \sum_{n\geq 1} \frac{nz^{n-1}}{(n+\cdot)^\tau}f_{\tau,z}\circ h_n + \H_{\tau,z}[\tfrac{\partial}{\partial z}f_{\tau,z}].  \end{array} \right. $$
We evaluate each line at~$(\tau,z)=(0,\hf)$ and apply the operator~$\P_{0,\hlf}$, which amounts to integrating against~$\dd\mu$ by point~(i). We obtain the claimed identities~$\tfrac{\partial}{\partial \tau}\lambda(0,\hf) = \int \log\dd\mu$ and~$\tfrac{\partial}{\partial z}\lambda(0,\hf) = \sum_{n\geq 1}n2^{1-n}=4$.
\end{proof}

\subsection{\texorpdfstring{The quasi-inverse  $(\Id-\H_{\tau,z})^{-1}$}{The quasi-inverse}}

We conclude the section by deducing the main properties of the quasi-inverse $(\Id-\H_{\tau,z})^{-1}$.

\begin{proposition}\label{prop:quasi-inverse}
For $\delta_1>0$ small enough there exist $\delta_2>0$, $0\leq \theta<1$ such that, for $\cV_1,\cV_2$ as in Proposition~\ref{prop:decomp-H}, one has that:
\begin{enumerate}[label=(\roman*)]
\item The operator $(\Id-\H_{\tau,z})^{-1}$ can be analytically continued to $(\tau,z)\in\cV_1$, $\lambda(\tau,z)\neq1$, via the identity
\begin{equation}
 (\Id-\H_{\tau,z})^{-1} = \frac{\lambda(\tau,z)}{(1-\lambda(\tau, z))} \P_{\tau,z} + (\Id - \N_{\tau,z})^{-1}.\label{eq:decomposition}
 \end{equation}
Moreover, $(\Id-\N_{\tau,z})^{-1}$ is holomorphic for $(\tau,z)\in\cV_1$.
\item $(\Id-\H_{\tau,z})^{-1}$ is holomorphic for $(\tau,z)\in\cV_2$.
\end{enumerate} 
\end{proposition}
\begin{proof} 
We start with (ii).
For~$(\tau,z)\in\cV_2$, by Proposition~\ref{prop:decomp-H} (v) we have~$\srd(\H_{\tau,z})\leq \theta < 1$. Therefore, for such $(\tau,z)$ and all $f\in H^\infty(\D)$ the series
$$ \sum_{k= 0}^\infty \H_{\tau,z}^k[f]= (\Id-\H_{\tau,z})^{-1}[f]$$
converges uniformly and defines an analytic function which depends on~$(\tau,z)\in\cV_2$ holomorphically.

To prove (i) first we notice that by Proposition~\ref{prop:measure} we can find a non-empty open set $\cU\subset \cV_1$ such that $|\lambda(\tau,z)|<1$ for all $(\tau,z)\in \cU$. Then, for $(\tau,z)\in\cU$ by Proposition~\ref{prop:decomp-H} (iii) one has
$$
(\Id-\H_{\tau,z})^{-1} =\sum_{k=0}^\infty \H_{\tau,z}^k=\sum_{k=1}^\infty \lambda(\tau,z)^k \P_{\tau,z}+\sum_{k=0}^\infty \N_{\tau,z}^k=\frac{\lambda(\tau,z)}{(1-\lambda(\tau, z))} \P_{\tau,z} + (\Id - \N_{\tau,z})^{-1}.
$$
Moreover, $\P_{\tau,z}$ and $(\Id - \N_{\tau,z})^{-1}$ are holomorphic for $(\tau,z)\in\cV_1$ and depend analytically on $\tau,z$. By Proposition~\ref{prop:measure} (ii) the set $\{(\tau,z)\in\cV_1\mid\lambda(\tau,z)\neq1\}$ is connected and so we can conclude by the identity principle. 
\end{proof}

\section{\texorpdfstring{The meromorphic continuation for $S_{\tau}(z)$}{The meromorphic continuation for the generating series}}\label{sec:analysis-Stz}

In this section we study the meromorphic continuation for $S_{\tau}(z)$, using Propositions~\ref{prop:Fsz-H-iter} and~\ref{prop:quasi-inverse}. As can be guessed from Proposition~\ref{prop:Fsz-H-iter}, the function~$S_\tau(z)$ will have a pole whenever~$\H_{\tau,z}$ has~$1$ as an eigenvalue. Therefore an important role is played by the quantity~$\rho(\tau)$, defined by the implicit relation~$\lambda(\tau,\rho(\tau))=1$. As will appear clearly in Section~\ref{sec:proof-th}, this produces a pole of~$S_\tau(z)$ at~$z=\rho(\tau)$ which will eventually give the dominant contribution in the estimate~\eqref{eq:estim-moments}.
% The function $U(\tau)$ given in Theorem~\ref{th:moments} will be equal to $-\log(2\rho(-\tau))$.

\subsection{Continuation and poles}

We start with the continuation and location of poles of~$S_\tau(z)$.

\begin{proposition}\label{prop:merom-F}
Given $\eta>0$ sufficiently small, there exists $0<c<\frac12$ such that $ S_{\tau}(z)$ is meromorphic for 
$$
(\tau,z)\in\{(\tau,z)\in\C^2\mid |\tau|\leq \eta, |z|\leq \hf+c\}.
$$
with a simple pole at $z=\rho(\tau)$ only, where $\rho$ is an analytic function in $|\tau|\leq \eta$ such that $\lambda(\tau,\rho(\tau))=1$ and $\rho(0)=\hf$. Finally, for $|\tau|\leq \eta$ we have
$$
\Res_{z=\rho(\tau)} S_{\tau}(z)=\frac{\P_{\tau,\rho(\tau)}[\1](1)}{(1-\rho(\tau))\frac{\partial}{\partial z}\lambda(\tau, \rho(\tau))}.
$$

\end{proposition}
\begin{proof}
The well-definedness and holomorphicity of~$\rho$ in a neighborhood of $\tau=0$ follows from the implicit function theorem since~$\lambda$ is holomorphic on~$\cV_1$ and~$\tfrac{\partial}{\partial z}\lambda(0,\hf)\neq 0$. Since $\lambda(0,\hf)=1$, this also yields~$\rho(0)=\hf$. 
The rest of the statement is then an immediate consequence of Proposition~\ref{prop:Fsz-H-iter} and Proposition~\ref{prop:quasi-inverse}.
\end{proof}

\subsection{\texorpdfstring{The derivatives of $\log(\rho(\tau))$}{The derivatives of log rho}}\label{sec:logrhosecond}

We now derive additional informations on the first two derivatives of~$\tau\mapsto \rho(\tau)$ at~$\tau=0$. This will lead to the expression~\eqref{eq:alpha-log} for the mean-value~$\alpha$ in Theorem~\ref{th:distrib}, and to positivity of the variance.

\begin{proposition}\label{prop:pole}
We have
\begin{equation}
\log(2\rho(0))=0, \qquad \alpha := (\log \rho)'(0) = -\frac12\int_{(0,1]}\log\dd\mu,\qquad (\log\rho)''(0)<0.\label{eq:def-alpha}
\end{equation}
More precisely, we have
\begin{align*}
(\log\rho)''(0) {}& = -\frac12\int_0^1 (\log x + \alpha\lfloor\tfrac1x\rfloor)^2\dd\mu(x) + \int_0^1\int_0^1 (\log x + \alpha\lfloor\tfrac1x\rfloor)\log(1+xy)\dd\mu(x)\dd\mu(y) \\ 
{}& = -\frac12\int_0^1 \Big(\log x + \alpha\lfloor\tfrac1x\rfloor + \int_0^1 \log\Big(\frac{1+y\{\tfrac1x\}}{1+yx}\Big)\dd\mu(y)\Big)^2 \dd\mu(x). \numberthis\label{eq:ellsecond-explicit}
\end{align*}
\end{proposition}

The fact that~$(\log\rho)''(0)<0$ is essential and it is equivalent to the non-degeneracy of the Gaussian law in Theorem~\ref{th:distrib}. The proof of this fact is not straightforward. As explained in~\cite[Lemma~7]{BV} (see also the proof of~(6b), page 343 there), it corresponds to a general result about convexity of pressure functions~\cite[Section~4.6]{Ruelle}. The existence of the explicit expression~\eqref{eq:ellsecond-explicit} contrasts with the context of statistics of continued fractions coefficients~\cite{BV}, where no explicit expression for the variance is known in general. The underlying construction is due to Benoist and Quint~\cite{BQ}.

For the proof of Proposition~\ref{prop:pole}, we mostly follow~\cite[Proposition~3.3]{Morris}. As in~\cite[Lemma~7]{BV}, the argument is mainly adapted from~\cite{Broise}; additionally, we will provide a simpler way to obtain non-positivity (see~\eqref{eq:simplif-psichi} below).

We will use repeatedly the following Lemma.
\begin{lemma}\label{lem:T-inv}
For~$f, g \in L^1(\mu)$, we have~$\H((f\circ \G) \times g ) = f \times(\H g)$ almost-everywhere.
\end{lemma}
\begin{proof}
For all~$x\in[0,1]$ and~$n\geq 1$, we have~$\G(h_n(x))=x$, and so~$((f\circ \G)\times g)\circ h_n = f\times (g\circ h_n)$. Summing over~$n\geq 1$ against~$2^{-n}$ yields the claimed equality.
\end{proof}

\begin{proof}[Proof of Proposition~\ref{prop:pole}]
In Proposition~\ref{prop:merom-F} we showed $\rho(0)=\hf$, whereas the value~$(\log\rho')(0)$ is easily obtained by differentiating~$\lambda(\tau,\rho(\tau))$ at~$\tau=0$ and using the values given in Proposition~\ref{prop:measure}.(ii).

We now wish to study the second derivative $(\log \rho)''(0)$, with the goal of proving that it is negative. For this purpose, it is convenient to consider a univariate function in place of~$\lambda(\tau, z)$. Let
$$ \ell(w) := \lambda(\alpha^{-1} w, \hf\e^w), $$
which is constructed in such a way that it is defined and analytic in a neighborhood of~$w=0$, and also~$\ell'(0)=0$. There is a simple relationship between~$\ell''(0)$ and the besought quantity $(\log\rho)''(0)$: differentiating twice the relation~$\lambda(\tau,\e^{\log \rho(\tau)})=1$ and evaluating at~$\tau=0$, we obtain
\begin{equation}\label{eq:relation-logrhosecond-ellsecond}
\begin{aligned}
(\log\rho)''(0) \rho(0) \tfrac{\partial}{\partial z}\lambda(0, \hf) {}& = -\big(\alpha^2 \rho(0)\tfrac{\partial}{\partial z}\lambda(0,\hf) + \tfrac{\partial^2}{\partial \tau^2}\lambda(0,\hf) \\ {}& \qquad \qquad \quad + \alpha \tfrac{\partial^2}{\partial \tau\partial z}\lambda(0,\hf) + \tfrac{\alpha^2}4\tfrac{\partial^2}{\partial z^2}\lambda(0,\hf)\big) \\
{}& = - \alpha^2 \ell''(0).
\end{aligned}
\end{equation}
Therefore, our task is to prove that~$\ell''(0)>0$. To obtain this, we write the eigenvalue equation relative to~$\ell(w)$, with the aim of differentiating twice, while gathering information along the way.

For~$w$ in a neighborhood of~$0$ we let~$\xi_w = f_{\frac w\alpha, \hlf \e^w}$ with the notation of Proposition~\ref{prop:decomp-H}. Recall that~$\xi_0=\bf1$. Then, the eigenvalue equation is
$$ \ell(w) \xi_w = \H_{\frac w\alpha, \hlf\e^w}[\xi_w] = \H[\e^{w\psi}\xi_w], $$
where
$$ \psi(x) = \begin{cases} \floor{1/x} + \alpha^{-1}\log(x) & (0<x\leq 1), \\ 0 & (x=0). \end{cases} $$
For future reference, we notice that~$\psi\in L^2(\mu)$.

By local uniform convergence, we may differentiate the above with respect to~$w$, obtaining
\begin{equation}
\ell'(w) \xi_w + \ell(w) \tfrac{\partial}{\partial w}\xi_w = \H[\e^{w\psi}(\psi \xi_w + \tfrac{\partial}{\partial w}\xi_w)].\label{eq:rel-deriv-1}
\end{equation}
Let~$\chi := [\tfrac{\partial}{\partial w}\xi_w]_{w=0} \in H^\infty(\D)$. Evaluating \eqref{eq:rel-deriv-1} at~$w=0$ and using~$\ell(0)=1$, $\ell'(0)=0$ and~$\xi_0=\1$ yields
\begin{equation}
\chi = \H[\psi + \chi].\label{eq:rel-chi}
\end{equation}

Before continuing the analysis, we focus on the function~$\chi$. In~\cite{BQ}, formula~(1.8), an explicit expression for solutions to cohomological equations related to~\eqref{eq:rel-chi} is obtained. Their construction is effectuated in our case as follows. Define
\begin{equation*}
x\mapsto \chi_1(x) := -\alpha^{-1} \int_0^1 \log(1+xy)\dd\mu(y) \label{eq:def-chi1-explicit}
\end{equation*}
It is obvious that~$\chi_1$ belongs to~$H^\infty(\D)$ and is real-valued. For all~$x\in[0, 1]$, we have
\begin{align*}
\H[\chi_1](x) {}& = -\alpha^{-1} \int_0^1 \sum_{n\geq 1} 2^{-n} \log\Big(1 + \frac{y}{x+n}\Big)\dd\mu(y) \\
{}& = -\alpha^{-1} \int_0^1\sum_{n\geq 1} 2^{-n} \log(x+y+n)\dd\mu(y) + \alpha^{-1} \sum_{n\geq 1} 2^{-n} \log(x+n) \\
{}& = -\alpha^{-1} \int_0^1 \log(x+1/y)\dd\mu(y) - \H[\alpha^{-1}\log](x) \\
{}& = \chi_1(x) - 2 - \H[\alpha^{-1}\log](x).
\end{align*}
Since~$\H[x\mapsto \floor{1/x}](x) = 2$, we obtain that~$\chi_1$ satisfies~\eqref{eq:rel-chi} on~$[0, 1]$, hence on~$\D$ by analytic continuation. In particular, we have $(\chi-\chi_1) = \H(\chi-\chi_1)$, and so, by Proposition~\ref{prop:decomp-H}.(ii),
\begin{equation}
\chi - \chi_1 \in \C\1.\label{eq:rel-chi-chi1}
\end{equation}

We may now continue our analysis. Differentiating again~\eqref{eq:rel-deriv-1} with respect to~$w$, we obtain
$$ \ell''(w) \xi_w + 2\ell'(w)\tfrac{\partial}{\partial w}\xi_w + \ell(w)\tfrac{\partial^2}{\partial w^2}\xi_w = \H[\e^{w\psi}(\psi^2\xi_w + 2\psi\tfrac{\partial}{\partial w}\xi_w + \tfrac{\partial^2}{\partial w^2}\xi_w)]. $$
Evaluating at~$w=0$ yields
\begin{equation}
\ell''(0)\1 + [\tfrac{\partial^2}{\partial w^2}\xi_w]_{w=0} = \H[\psi^2 + 2\psi \chi + (\tfrac{\partial^2}{\partial w^2}\xi_w)_{w=0}].\label{eq:pre-rel-deriv-2}
\end{equation}
Here, note that the function~$\psi^2 + 2\psi \chi + (\tfrac{\partial^2}{\partial w^2}\xi_w)_{w=0}$ is in~$L^1([0,1],\mu)$. We integrate both sides of~\eqref{eq:pre-rel-deriv-2} against~$\dd\mu$, in order to eliminate the terms~$\frac{\partial^2}{\partial w^2}\xi_w$. Since~$\int_{[0,1]} \H[f]\dd \mu = \int_{[0,1]} f \dd\mu$ for any~$f\in L^1(\mu)$, we arrive at
\begin{align*}
\ell''(0) {}& = \int_{[0,1]} (\psi^2 + 2\psi \chi)\dd\mu \\
{}& = \int_{[0, 1]} (\psi^2 + 2\psi \chi_1) \dd\mu    \numberthis\label{eq:ellsecond-prelim}
\end{align*}
where we have used~\eqref{eq:rel-chi-chi1} together with the value~$\int \psi\dd\mu = 0$. Define now~$\psih := \tau + \chi_1 - \chi_1 \circ \G$; this definition is motivated by analogy with the proof of Proposition~3.3 of~\cite{Morris}. Note that~$\psih$ is real-valued. Using~Lemma~\ref{lem:T-inv}, we have
$$ \int_{[0,1]} \chi_1^2\circ \G\dd\mu = \int_{[0,1]} \chi_1^2\cdot \H[\1]\dd\mu=\int_{[0,1]} \chi_1^2\dd\mu, $$
$$ \int_{[0,1]} (\psi + \chi_1) (\chi_1\circ \G)\dd\mu = \int_{[0,1]}\H[\psi+\chi_1] \chi_1 \dd\mu = \int_{[0,1]} \chi_1^2\dd\mu. $$
In the last step, we used the fact that equation~\eqref{eq:rel-chi} holds with~$\chi$ replaced by~$\chi_1$. We deduce
\begin{equation}\label{eq:simplif-psichi}\begin{aligned}
\int_{[0,1]}(\psih)^2\dd\mu {}& = \int_{[0,1]}(\psi+\chi_1)^2\dd\mu - 2\int_{[0,1]}(\psi+\chi_1)(\chi_1\circ \G)\dd\mu + \int_{[0,1]}(\chi_1\circ \G)^2\dd\mu \\
{}& = \int_{[0,1]} (\psi + \chi_1)^2 \dd\mu - \int_{[0,1]} \chi_1^2\dd\mu \\
{}& = \ell''(0).
\end{aligned}\end{equation}
Now, we note that~$\psi(x) = \floor{1/x} + \alpha^{-1}\log(x)$ tends to~$+\infty$ as~$x\to0$, whereas~$\chi_1\in H^\infty(\D)$ is a bounded function. We may therefore find~$\beta>0$ such that~$|\psih(x)|\geq 1$ for all~$x\in (0,\beta)$. We deduce that~$\ell''(0) = \int_{[0,1]}(\psih)^2\dd\mu \geq \mu((0, \beta)) = \Psi(\beta) > 0$ as required.

Since, by~\eqref{eq:relation-logrhosecond-ellsecond}, we have~$(\log\rho)''(0) = -\frac12\alpha^2 \ell''(0)$, the explicit formula~\eqref{eq:ellsecond-explicit} follows from~\eqref{eq:ellsecond-prelim} and \eqref{eq:simplif-psichi}.
\end{proof}

\section{Proof of Theorem~\ref{th:moments}}\label{sec:proof-th}

By the definitions~\eqref{eq:def-MN-F} and Cauchy's formula, for $|\tau|\leq\eta$ with $\eta$ sufficiently small we have
$$ 2^N\E_N[(\S_N)^{-\tau}] = \frac1{2\pi i}\oint_{|z|=\frac14} z^{-N} S_\tau(z) \frac{\dd z}{z}, $$
where the circle is oriented counter-clockwise. By Proposition~\eqref{prop:merom-F} and the residue theorem for some $0<c<\hf$ we have
$$ 2^N\E_N((\S_N)^{-\tau}) = \frac{R(\tau)}{\rho(\tau)^N}+ \frac1{2\pi i}\oint_{|z|=\frac12+c} z^{-N} S_{\tau}(z) \frac{\dd z}{z}, $$
where the circle is oriented counter-clockwise and
$$ R(\tau) = \frac{\P_{s,\rho(\tau)}[\1](1)}{\rho(\tau)(1-\rho(\tau))\frac{\partial}{\partial z}\lambda(\tau, \rho(\tau))}.$$
Note that~$R(\tau)$ is analytic for~$|\tau|\leq \eta$ and~$R(0) = 1$. We bound the integral trivially
$$
\frac1{2\pi i}\oint_{|z|=\frac12+c} z^{-N} S_{\tau}(z) \frac{\dd z}{z}\ll (\hf + c)^{-N} \ll 2^N (1+2c)^{-N},
$$
uniformly for $|\tau|\leq \eta$. At the possible cost of reducing~$\eta$, we may write~$R(\tau) = \exp V(\tau)$, and similarly~$2\rho(\tau) = \exp \{-U(\tau)\}$, for two functions~$U$ and~$V$ which are holomorphic for~$|\tau|\leq \eta$. By Proposition~\ref{prop:pole} we also have $U(0)=V(0)=0$, $U'(0)=\frac12\int_{(0,1]}\log\dd\mu$, $U''(0)>0$. In the same range of~$\tau$, we conclude that
\begin{equation*}
\E_N((\S_N)^{-\tau}) = \exp\big(N U(\tau) + V(\tau))\big\{1 + O((1+2c)^{-N})\big\}\label{eq:estim-final-moment}
\end{equation*}
for some~$c>0$. This is the statement of Theorem~\ref{th:moments}, up to changing~$\tau$ to~$-\tau$. Theorem~\ref{th:distrib} follows at once using Hwang's quasi-power theorem, for which we refer to~\cite[Lemma~IX.1]{FS} and to the papers~\cite{Hwang1, Hwang2}. We note for future reference that the variance is related to~$\rho$ by
\begin{equation}
\sigma^2 = -(\log \rho)''(0).\label{eq:rel-variance-logrho}
\end{equation}

\begin{remark}
\begin{itemize}
\item The foregoing computations may be seen as a materialization of the link between the spectra of transfer operators associated with the Farey and Gauss map. There have been many works on this topic: see \emph{e.g.}~\cite{PS},~\cite{Isola},~\cite{Prellberg}, and also~\cite{FMT} for works on the numerical aspect. Another useful reference is the introduction of~\cite{Dodds}.

In the above, by using the Cauchy formula on the generating series~\eqref{eq:def-MN-F}, we have avoided completely discussing the Farey map and its associated tranfer operator
$$ {\mathbb L}[f](x) = \frac12 f\Big(\frac1{x+1}\Big) + \frac12 f\Big(\frac x{x+1}\Big), $$
whose functional analysis is made difficult by the fixed point with derivative~$1$ (neutral fixed point) of the Farey map at~$0$. Instead, we extracted the precise information we needed, which is the exponential convergence for iterates of the function~$\1$.

% \item If~$\delta$ denotes the spectral gap of the operator~$\mathbb{L}$ acting on the function space spanned by iterates of~$\1$, then any~$\theta\in(0, \delta/2)$ is admissible in the error term of Theorems~\ref{th:distrib} and \ref{th:moments}.

\end{itemize}
\end{remark}

\section{\texorpdfstring{The mean-value~$\alpha$}{The mean-value alpha}}\label{sec:alpha}

In this section, we detail and prove Theorem~\ref{th:alpha}.

\subsection{Variants of the integral expression}

\begin{lemma}\label{lem:alpha-log1px-log}
For the quantity~$\alpha$ defined in~\eqref{eq:def-alpha}, we have~$\displaystyle\alpha =  \int_0^1 \log(1+x)\dd\mu(x)$.
\end{lemma}
\begin{proof}
We remark that
\begin{align*}
\log(1+x) + \H[t\mapsto \log(1+t)](x) {}& = \log(1+x) + \sum_{n\geq 1}2^{-n}\log\Big(\frac{1+n+x}{n+x}\Big) \\
{}& =  \sum_{n\geq 1} 2^{-n} \log(n+x) \\
{}& = -\H[\log](x)
\end{align*}
by splitting the logarithm into a difference. Integrating against~$\dd\mu(x)$ and using~\eqref{eq:eigenmeasure} yields the claimed equality in the form $2\int_0^1 \log(1+x)\dd\mu(x) = - \int \log \dd\mu$.
\end{proof}

We can now justify assertion~(i). Using Lemma~\ref{lem:alpha-log1px-log}, we see that
$$ \alpha = \int_0^1 \log(1+x)\dd\mu(x) = \int_0^1 \log(2-x) \dd\mu(x) = \log 2 + \int_0^1\log(1-\tfrac x2)\dd\mu(x) $$
where we have used the symmetry~$\dd\mu(1-x)=-\dd\mu(x)$. Expanding the logarithm into a power series, we recover the definition~\eqref{bacher_alpha} used in~\cite{Bacher}.

Note that Conjectures 8.1 and 8.2 of~\cite{Bacher}, which are concerned with similar identities, can be proven along the same lines.

\subsection{\texorpdfstring{Lipschitz points of~$\Psi$}{Lipschitz points of Psi}}

In the paper~\cite{Kinney}, Kinney studies the set of Lipschitz points of the Minkowski function~$\Psi$. There has been many consecutive works on this topic, see \emph{e.g.}~\cite{MZ,KS-Minkowski,KS-Diophantine}.

Theorem~1 of~\cite{Kinney} states that there exists a set~$V\subset[0, 1]$ with~$\mu(V)=1$ and Hausdorff dimension~$\beta$ on which the Lipschitz exponent of~$\Psi$ is~$\beta$. The value of $\beta$ is recognized, using Lemma~\ref{lem:alpha-log1px-log}, to be precisely
$$ \beta = \frac{\log 2}{2\alpha}. $$
Theorem~\ref{th:distrib} may be used to recover Kinney's result. We first give the following lemma, which is easily proved by induction.
\begin{lemma}
For all~$N\geq 0$ and~$m\in\{0, \dotsc, 2^N-1\}$, we have
\begin{equation}
\frac{s(m+1)}{s(2^N+m+1)} - \frac{s(m)}{s(2^N+m)} = \frac1{s(2^N+m)s(2^N+m+1)}.    \label{eq:gap-phi}
\end{equation}
\end{lemma}
By equations~\eqref{eq:gap-phi},~\eqref{eq:link-phi-stern}, and Theorem~\ref{th:distrib} with~$t=\pm \log N$ (say), we obtain
\begin{equation}
\log\Big|\Phi\Big(\frac{m+1}{2^N}\Big) - \Phi\Big(\frac{m}{2^N}\Big)\Big| = -2 N \alpha + O(\sqrt{N}\log N)\label{eq:diffLogPhi-2alpha}
\end{equation}
for all~$m$ in a set~$R_N \subset \{0, \dotsc, 2^N-1\}$ of cardinality~$\left|R_N\right| \geq 2^N - O(2^N/N^2)$. We then define
$$ U = \bigcup_{N_0 \geq 0} \bigcap_{N\geq N_0} \bigcup_{\substack{1\leq m \leq 2^N-1 \\(m-1, m, m+1) \in R_N^3}} \Big[\frac{m}{2^N}, \frac{m+1}{2^N}\Big]. $$
Since~$\sum_{N\geq 0} (1-|R_N|/2^N) < \infty$, the set~$U$ has Lebesgue measure~$1$. Now let~$x\in U$ and~$\ee>0$ be fixed. Let~$N_0 = N_0(\ee) \geq 1/\ee^3$ be an integer, such that~$x \in \cap_{N\geq N_0} \cup_{(m-1, m, m+1) \in R_N^3} [\frac{m}{2^N}, \frac{m+1}{2^N}]$. Then for~$0<h\leq 2^{-N_0}$, whenever~$2^{-1-N}\leq h \leq 2^{-N}$, we have
$$ \Big[\frac m{2^{N+1}}, \frac{m+1}{2^{N+1}}\Big] \subset [x-h, x+h] \subset \Big[\frac{m'-1}{2^N}, \frac{m'+2}{2^N}\Big] $$
for some integers~$(m, m')$ satisfying~$m'\in R_{N+1}$ and~$(m-1, m, m+1)\in R_N^3$. From~\eqref{eq:diffLogPhi-2alpha}, we deduce
$$ \e^{-N(2\alpha + O(\ee))} \leq \left|\Phi(x+h) - \Phi(x-h)\right| \leq 3 \e^{-N(2\alpha - O(\ee))}. $$
Since both the left-hand side and the right-hand side are of the order~$h^{2\alpha/\log 2 + O(\ee)}$, we conclude that the Lipschitz exponent of~$\Phi$ at any~$x\in U$ is equal to~$2\alpha/\log 2$. Additionally,~$\Phi$ and~$\Psi$ are inverse bijections, therefore the Lipschitz exponent of~$\Psi$ at any~$x\in V := \Phi(U)$ is~$\log 2/(2\alpha)$. Because~$\mu(V)$ is equal to the Lebesgue measure of~$U$, which is $1$, we obtain that on $V$ the Lipschitz exponent of~$\Psi$ is~$\beta$. 

Finally, one has
$$
V=\bigcup_{N_0 \geq 0} \bigcap_{N\geq N_0} \bigcup_{\substack{1\leq m \leq 2^N-1 \\(m-1, m, m+1) \in R_N^3}} \bigg[\frac{s(m+1)}{s(2^N+m+1)} , \frac{s(m)}{s(2^N+m)}\bigg]
$$
and, for any fixed $u\geq0$
\begin{align*}
\sum_{\substack{1\leq m \leq 2^N-1 \\(m-1, m, m+1) \in R_N^3}} \bigg| \frac{s(m)}{s(2^N+m)}-\frac{s(m+1)}{s(2^N+m+1)} \bigg|^{u}&=2^N(1+O(N^{-2}))e^{ -2 u N \alpha + o(N)}\\[-1em]
&=e^{ (\log 2-2 u\alpha) N +o(N)}.
\end{align*}
by~\eqref{eq:diffLogPhi-2alpha}. The fact that $V$ has Hausdorff dimension $\leq\beta=\frac{\log2}{2\alpha}$ then follows immediately. To prove the opposite inequality, we proceed as in~\cite{Kinney} and observe that since $\Psi$ has Lipschitz exponent~$\beta$, then for all $\varepsilon>0$ there exist $h>0$ and a subset $V'\subseteq V$ with $\mu(V')\geq\frac12$, say, such that
\begin{equation}\label{dm}
\Psi(x+h',x-h')=\mu\big([x-h',x+h']\big)<(2h')^{\beta-\varepsilon}
\end{equation}
for all $x\in V$ and all $0<h'\leq h$. Now let $\{I_i\mid i\in I\}$ be a cover of $V'$ with intervals $I_i$ satisfying $|I_i|\leq h$. Then, taking $x_i\in V'\cap I_i$ we have $I_i\subseteq I_i':=[x_i-|I_i|,x_i+|I_i|]$ and $|I_i'|\leq 2|I_i|$. In particular, by~\eqref{dm},
$$
\frac12\leq \mu(V')\leq \sum_{i\in I}\mu(I'_i)<\sum_{i\in I}(2|I'_i|)^{\beta-\varepsilon} \leq \sum_{i\in I}(4|I'_i|)^{\beta-\varepsilon}.
$$
This implies that the Hausdorff dimension of $V'$ (and henceforth that of $V$) is $\geq\beta-\varepsilon$. Since $\varepsilon$ is arbitrary, the result follows.

\subsection{\texorpdfstring{The constant~$2\alpha$ as a metric entropy}{The constant 2alpha as a metric entropy}}

We now explain how one may interpret formula~\eqref{eq:expect} as the computation of a certain partition entropy. This may be used to derive another expression for~$\alpha$.

Recall that the Farey map~$\F$ was defined at~\eqref{eq:def-Farey} (see Figure~\ref{fig:graph-gauss-0}). Also, in this section only, we denote by~$\nu$ the Lebesgue measure on~$[0, 1]$. We start with the partition
$$ [0,1) = \bigvee_{k=0}^{N} F^{-k}[0, 1) = \bigcup_{m=0}^{2^N-1} J_{m,N}, $$
where~$J_{m,N}$ are consecutive segments; explicitly, we have~$J_{m,N} =[\Phi (\tfrac{m}{2^N}), \Phi (\tfrac{m+1}{2^N}))$. By~\eqref{eq:gap-phi} we have
$$ \nu(J_{m,N}) = \frac1{s(2^N+m)s(2^N+m+1)}, \qquad \mu(J_{m,N})=2^{-N}. $$
Since~$s(2^{N+1})=s(2^N)$, we deduce
\begin{equation}
-\frac12 \sum_{m=0}^{2^N-1} \mu(J_{m,N}) \log \nu(J_{m,N}) = \frac2{2^N} \sum_{m=0}^{2^N-1} \log s(2^N + m) .\label{eq:alpha-partition}
\end{equation}
The left-hand side can be seen as a relative entropy with respect to the partition~$(J_{m,N})_m$. As~$N$ tends to infinity, we can adapt the proof of Rokhlin's formula~\cite[Theorem~12.10]{PY} to show that
\begin{equation}
\label{eq:alpha-partition-limit}
-\frac1{N} \sum_{m=0}^{2^N-1} \mu(J_{m,N}) \log \nu(J_{m,N}) \longrightarrow \int_0^1 \log |\F'(x)|\dd\mu(x).
\end{equation}
To see this, note first that the R\'enyi condition~\cite[formula~(12.2)]{PY} holds in our case in the weaker form
$$ \sup_{(x, y) \in [0, 1]^2} \Big|\frac{\psi'(x)}{\psi'(y)}\Big| \leq N^{O(1)} $$
for any inverse branch~$\psi$ of~$\F^N$; the steps~(i) and~(ii) on page~133 of~\cite{PY} are easily carried out and yield~\eqref{eq:alpha-partition-limit}.

For the equations~\eqref{eq:expect}, \eqref{eq:alpha-partition} and~\eqref{eq:alpha-partition-limit} to agree, the equality~\eqref{eq:alpha-entropy} must hold. Indeed, since~$\log|F'(x)| = -2\log(\max\{x, 1-x\})$ and~$\dd\mu(1-x)=-\dd\mu(x)$, we deduce
\begin{equation*}
\int_0^1\log\left|\F'(x)\right|\dd\mu(x) = -4\int_{1/2}^1 (\log x) \dd\mu(x) = 2\int_0^1 \log(1+x)\dd\mu(x)
\end{equation*}
where we have used the equality~$\H[t\mapsto \1_{t>1/2} \log t](x) = -\hf \log(1+x)$. By Lemma~\ref{lem:alpha-log1px-log}, the claimed equality~\eqref{eq:alpha-entropy} follows.

\subsection{\texorpdfstring{The constant~$\alpha$ as a Lyapunov exponent}{The constant alpha as a Lyapunov exponent}}

We have already noted in~\eqref{eq:stern-product-matrices} that the quantity~$\log \S_N$ may be expressed as a random product in the following way: define, as earlier, the measure~$\eta$ on~$GL_2(\R)$ by~$\eta = \frac12\delta_{A_0}  + \frac12\delta_{A_1}$, where $A_0 = \left(\begin{smallmatrix} 1 & 1 \\ 0 & 1 \end{smallmatrix}\right)$ and~$A_1 = \left(\begin{smallmatrix} 1 & 0 \\ 1 & 1 \end{smallmatrix}\right)$. Let also~$v = \left( \begin{smallmatrix} 1 \\ 1 \end{smallmatrix} \right)$, and~$f(x) = x \cdot \left(\begin{smallmatrix} 0 \\ 1 \end{smallmatrix}\right)$. Then
$$ \log \S_N = \log |f(g_N \cdots g_1 v)| $$
where~$g_j\in GL_2(\R)$ are taken independently at random according to~$\eta$ (which simply means $g_j=A_0$ or~$A_1$ with equal probability). Then, by definition, the mean-value $\alpha$ of $\log \S_N $ should coincide with the first Lyapunov exponent~$\lambda_1$ of~$\eta$; by a formula due to Furstenberg~\cite[Theorem~8.5]{Furstenberg}, we have
\begin{equation}
\lambda_1 = \int_{GL_2(\R) \times \P(\R^2)} \log\frac{|f(gv)|}{|f(v)|} \dd \eta(g) \dd\xi(v) \label{eq:expr-lambda1}
\end{equation}
for any measure~$\xi$ on~$\P(\R^2)$ with~$\eta \ast \xi = \xi$. Let us now explain how this expression may be used to recover the formula~\eqref{eq:alpha-log}, which amounts to finding an admissible measure~$\xi$. Paramatrizing~$\P(\R^2) = \R \cup \{\infty\}$, consider the Minkowski measure~$\mu$ as being defined on all~$\R\cup\{\infty\}$ but supported on~$[0,1]$. Let~$T = \left(\begin{smallmatrix} & 1 \\ 1 & \end{smallmatrix}\right)$; note that~$TA_1 = A_0T$. Moreover, the Minkowski measure satisfies
$$ \mu = \tfrac12(A_1^* \mu + (A_1T)^* \mu). $$
Indeed, we have~$\H^* \mu = \mu$, but from the definition of~$\H$,
$$ \H^* \mu = \sum_{n\geq 1} 2^{-n} (TA_0^n)^* \mu = \tfrac12 (A_1T)^* \mu + \tfrac12 A_1^* \H^*\mu $$
since~$TA_0^n = A_1TA_0^{n-1}$. Define~$\xi$ as
$$ \xi = \tfrac12 \mu + \tfrac12 (T^* \mu). $$
A small computation, using~$A_0 = TA_1T$ and~$A_0T = TA_1$, yields
\begin{align*}
\eta \ast \xi {}& = \tfrac14(A_1^* \mu + (A_1T)^*\mu + (TA_1T)^* \mu +  (TA_1)^* \mu) \\
{}& = \tfrac12 \mu + \tfrac12 T^* \mu = \xi
\end{align*}
so the conditions for the integral~\eqref{eq:expr-lambda1} to hold are met. The integral on the right-hand side of~\eqref{eq:expr-lambda1} then reads
\begin{align*}
\lambda_1 = {}& \int_{\R_+} \tfrac12 \log|f(A_0\! \left(\begin{smallmatrix} x \\ 1 \end{smallmatrix}\right)) \  f(A_1\! \left(\begin{smallmatrix} x\\ 1 \end{smallmatrix}\right))| \dd \xi(x) \\
= {}& \frac12\int_0^\infty \log(1+x) \dd\xi(x) \\
= {}& \frac14\int_0^1 (\log(1+x) + \log(1+x^{-1})) \dd\mu(x) %\\
%= {}& \frac12\int_0^1 \log(1+x)\dd\mu(x) - \frac14 \int_0^1 (\log x)\dd\mu(x)
\end{align*}
which evaluates to~$\alpha$ by~\eqref{eq:alpha-log} and Lemma~\ref{lem:alpha-log1px-log}.

\subsection{Explicit expression for the variance}\label{sec:logrhosecond-explicit}

We recall that the variance in Theorem~\ref{th:distrib} is related to~$\tau\mapsto \rho(\tau)$ by formula~\eqref{eq:rel-variance-logrho}. The expression stated in Theorem~\ref{th:alpha}.\ref{list:logrho-explicit} immediately follows using formula~\eqref{eq:ellsecond-explicit}. Note that the alternative expression
\begin{equation*}
\sigma^2 = \frac12\int_0^1 (\log x + \alpha\lfloor\tfrac1x\rfloor)^2\dd\mu(x) - \int_0^1\int_0^1 (\log x + \alpha\lfloor\tfrac1x\rfloor)\log(1+xy)\dd\mu(x)\dd\mu(y)
\end{equation*}
is particularly well-suited for numerical computation on the lines of~\cite{Bacher}.

\subsubsection*{Acknowledgements}

The authors are grateful to B. Vall\'ee and to the anonymous referee for numerous insightful remarks on earlier versions of this manuscript. SD thanks \'E. de Panafieu, J. Emme, P. Hubert, S. Troubetzkoy and G. Merlet for helpful discussions. Part of this work was done during a visit of SD at University of Genova, and a visit of LS at Aix-Marseille University. Both institutions are thanked for their support. The work of the SB is partially supported by  PRIN ``Number Theory and Arithmetic Geometry''.
% The work of LS is partially supported by the ANR-FWF Project ``MuDeRa''.

% BIBLIOGRAPHY
% ===========================================

% \bibliographystyle{amsalpha}
% \bibliography{distrib}

\providecommand{\bysame}{\leavevmode\hbox to3em{\hrulefill}\thinspace}
\providecommand{\MR}{\relax\ifhmode\unskip\space\fi MR }
\providecommand{\MRhref}[2]{%
  \href{http://www.ams.org/mathscinet-getitem?mr=#1}{#2}
}
\providecommand{\href}[2]{#2}

\end{document}